\documentclass{amsart}

\usepackage{amscd,amsmath,amssymb,amsfonts,bbm, calligra}
\usepackage[T1]{fontenc}
\usepackage{lmodern}
\usepackage{mathtools}
\usepackage{enumerate}
\usepackage{multicol}
\usepackage[all]{xy}
\usepackage{mathrsfs}
\usepackage{footmisc}

\newtheorem{thm}{Theorem}[section]
\newtheorem{prop}[thm]{Proposition}
\newtheorem{conj}[thm]{Conjecture}

\newtheorem{lem}[thm]{Lemma}

\theoremstyle{definition}

\newtheorem{assumption}[thm]{Assumption}
\theoremstyle{remark}

\numberwithin{equation}{section}

\DeclareMathOperator{\res}{res}
\DeclareMathOperator{\cl}{cl}

\DeclareMathOperator{\et}{\mathrm{\acute{e}t}}
\DeclareMathOperator{\van}{van}

\DeclareMathOperator{\Br}{Br}

\DeclareMathOperator{\Ker}{Ker}
\DeclareMathOperator{\Pic}{Pic}
\DeclareMathOperator{\Sing}{Sing}

\DeclareMathOperator{\Id}{Id}
\DeclareMathOperator{\tors}{tors}

\DeclareMathOperator{\Gal}{Gal}

\DeclareMathOperator{\CH}{CH}
\DeclareMathOperator{\per}{per}
\DeclareMathOperator{\ind}{ind}

\DeclareMathOperator{\Tr}{Tr}

\DeclareMathOperator{\sExt}{\mathscr{E}\text{\kern -3pt {\calligra\large xt} }}
\DeclareMathOperator{\RHom}{\mathrm{R}\mathscr{H}\text{\kern -3pt {\calligra\large om}}\,}
\DeclareMathOperator{\sHom}{\mathscr{H}\text{\kern -3pt {\calligra\large om}}\,}
\DeclareMathOperator{\cupp}{\mkern-1mu\smile\mkern-1mu}

\newcommand{\congr}[3]{#1 \equiv #2 \!\!\!\!\mod #3}
\newcommand{\congru}[3]{#1 \equiv #2 \!\!\mod #3}
\newcommand{\LL}{{\mathbf L}}
\newcommand{\ci}{\mathcal{C}^{\infty}}
\newcommand{\co}{\mathcal{C}^{0}}

\def\ssetminus{\mkern-5mu\setminus\mkern-5mu}
\def\sto{\mkern-4mu\to\mkern-4mu}
\def\soplus{\mkern-4mu\oplus\mkern-4mu}
\def\s={\mkern-3mu=\mkern-3mu}

\def\splus{\mkern-1mu +\mkern-1mu}

\def\RR{\mathrm{R}}
\def\vep{\varepsilon}

\def\cA{\mathcal{A}}
\def\cO{\mathcal{O}}
\def\cH{\mathcal{H}}
\def\cL{\mathcal{L}}
\def\cM{\mathcal{M}}
\def\cN{\mathcal{N}}
\def\cT{\mathcal{T}}

\def\cK{\mathcal{K}}
\def\cI{\mathcal{I}}
\def\cU{\mathcal{U}}

\def\cX{\mathcal{X}}

\def\sF{\mathscr{F}}
\def\sG{\mathscr{G}}
\def\sL{\mathscr{L}}
\def\sQ{\mathscr{Q}}

\def\kX{\mathfrak X}
\def\bkX{\overline{\mathfrak X}}

\def\bK{\mathbb K}
\def\Z{\mathbb Z}
\def\C{\mathbb C}
\def\Q{\mathbb Q}
\def\R{\mathbb R}
\def\bN{\mathbb N}
\def\bS{\mathbb S}
\def\bP{\mathbb P}

\def\bH{\mathbb H}
\def\SI{\mathbb{S}^1}

\def\G{\mathbf G}
\def\bmu{\boldsymbol\mu}

\def\bT{\overline{T}}
\def\bp{\overline{p}}

\def\hrho{\hat \rho}
\def\hp{\hat p}

\def\he{\hat e}
\def\hu{\hat u}
\def\hS{\hat S}
\def\hep{\hat\varepsilon}
\def\hW{\widehat W}
\def\hV{\widehat V}
\def\hSI{\widehat{\mathbb{S}^1}}

\def\trho{\tilde \rho}

\def\tu{\tilde u}
\def\tR{\widetilde R}
\def\tep{\tilde\varepsilon}
\def\tSI{\widetilde {\mathbb{S}^1}}

\def\tX{\widetilde X}
\def\tZ{\widetilde \Z}
\def\tW{\widetilde W}
\def\talpha{\tilde{\alpha}}

\begin{document}

\title[The period-index problem for real surfaces]
{The period-index problem for real surfaces}

\author{Olivier Benoist}
\address{D\'epartement de math\'ematiques et applications, \'Ecole normale sup\'erieure,
45 rue d'Ulm, 75230 Paris Cedex 05, France}
\email{olivier.benoist@ens.fr}

\renewcommand{\abstractname}{Abstract}
\begin{abstract}
We study when the period and the index of a class in the Brauer group of the function field of a real algebraic surface coincide. We prove that it is always the case if the surface has no real points (more generally, if the class vanishes in restriction  to the real points of the locus where
 it is well-defined), and give a necessary and sufficient condition for unramified classes. 
As an application, we show that the $u$-invariant of the function field of a real algebraic surface is equal to $4$, answering questions of Lang and Pfister.
Our strategy relies on a new Hodge-theoretic approach to de Jong's period-index theorem on complex surfaces.
\end{abstract}
\maketitle

\section*{Introduction}\label{intro}

\subsection{The period-index problem}

  Let $K$ be a field, and let $\Br(K)$ be its Brauer group. The period $\per(\alpha)$ of $\alpha\in \Br(K)$ is its order in $\Br(K)$ and its index $\ind(\alpha)$ is the smallest (equivalently, the gcd) of the degrees of the finite field extensions $L/K$ over which $\alpha$ vanishes. In general, $\per(\alpha)\mid\ind(\alpha)$, and these invariants have the same prime divisors.
  Finding further constraints on the period and the index
is the so-called period-index problem (see \cite{BourCT} for an account of this question).

Two outstanding results are de Jong and Lieblich's theorems on function fields of surfaces over algebraically closed (\cite{deJong}, see also \cite[Theorem 4.2.2.3]{Lieblichtwisted}) or finite fields \cite[Theorem 1.1]{Lieblich}
(see \cite{padic} for results on function fields of $p$-adic surfaces).

\begin{thm}[de Jong] 
\label{pidJ}
Let $S$ be a connected smooth projective surface over an algebraically closed field $k$. If $\alpha\in\Br(k(S))$, then $\ind(\alpha)=\per(\alpha)$.
\end{thm}

\begin{thm}[Lieblich]
\label{piLieblich}
Let $S$ be a connected smooth projective surface over a finite field $k$.  If $\alpha\in\Br(k(S))$, then $\ind(\alpha)\mid\per(\alpha)^2$.
\end{thm}

A general guideline
 is that if $K$ has cohomological dimension $\delta$, one might hope that $\ind(\alpha)\mid\per(\alpha)^{\delta-1}$ for every $\alpha\in\Br(K)$. Theorems \ref{pidJ} and \ref{piLieblich} fit into this philosophy, but Merkurjev has constructed convoluted counterexamples \cite[\S 3]{Merkurjev}. The case of $K=\C(x,y,z)$, that has cohomological dimension $3$, is wide open.

In this paper, relying on a new Hodge-theoretic approach to de Jong's theorem (see \S\ref{strategy} and Section \ref{secdJ}), we investigate the case of function fields $K$ of real algebraic surfaces. They may have infinite cohomological dimension, but have virtual cohomological dimension $2$ (i.e. $K[\sqrt{-1}]$ has cohomological dimension $2$).

\subsection{Function fields of real surfaces}
\label{nicestatements}

Let $S$ be a connected smooth projective surface over $\R$, and $\alpha\in \Br(\R(S))$ be a Brauer class.
De Jong's Theorem~\ref{pidJ} and a norm argument show that $\ind(\alpha)=\per(\alpha)$ or $\ind(\alpha)=2\per(\alpha)$.
We show that the equality $\ind(\alpha)=\per(\alpha)$ always holds if $S$ has no real points.

\begin{thm}\label{pisanspoint}
Let $S$ be a connected smooth projective surface over $\R$ such that $S(\R)=\varnothing$. If $\alpha\in\Br(\R(S))$, then $\ind(\alpha)=\per(\alpha)$.
\end{thm}

This gives new examples of fields of  cohomological dimension $2$ (by \cite[Proposition 1.2.1]{CTParimala}), such as $K=\R(x,y,z\mid x^2+y^2+z^2=-1)$, on which period and index coincide. As we explain in \S\ref{parLang}, Theorem \ref{pisanspoint} was predicted by a conjecture of Lang. 

In general, the class $\alpha$ belongs to the subgroup $\Br(U)\subset \Br(\R(S))$ for some open subset $U\subset S$ (see \S\ref{parBrauer}).
 Theorem \ref{pisanspoint} generalizes to the case when $\alpha$ vanishes in restriction to the real points of $U$. We also explain in \S \ref{parLang} that this statement had been conjectured by Pfister.

\begin{thm}
\label{piavecpoint}
Let $U$ be a connected smooth surface over $\R$ and let $\alpha\in \Br(U)\subset\Br(\R(U))$ be such that for every $x\in U(\R)$, $\alpha|_x=0\in\Br(\R)$. Then $\ind(\alpha)=\per(\alpha)$.
\end{thm}

We refer to \S\ref{parLang} for applications of Theorems \ref{pisanspoint} and \ref{piavecpoint} to the arithmetic of function fields of real varieties, that were the main motivation for this work.

\subsection{Unramified classes}
Brauer classes that belong to the subgroup $\Br(S)$ of $\Br(\R(S))$ are said to be unramified, and are often better behaved (over finite fields, see \cite[Theorem 4.3.1.1]{Lieblichtwisted}). We compute their index entirely.

For $\alpha\in\Br(S)\subset  \Br(\R(S))$, define $\Theta:=\{x\in S(\R)\ |\ \alpha|_x\neq 0\in\Br(\R)\}$. It is a union of connected components of $S(\R)$.
As explained in \S\ref{parBrauer}, there is a short exact sequence (\ref{Brlift}):
\begin{equation}
\label{BrliftS}
0\to \Pic(S)/2\to H^2_G(S(\C),\Z/2)\to \Br(S)[2]\to 0,
\end{equation}
where $H^2_G(S(\C),\Z/2)$ is an equivariant cohomology group with respect to the action of $G=\Gal(\C/\R)\simeq\Z/2$ on $S(\C)$.
 In (\ref{ithfactor2}) of \S\ref{subsecrest}, we define
 a morphism $H^2_G(S(\C),\Z/2)\to H^1(S(\R),\Z/2)$ denoted by $\xi\mapsto [\xi]_1$. We may now
state:

\begin{thm}\label{pinr}
Let $S$ be a connected smooth projective surface over $\R$, and let $\alpha\in\Br(S)\subset\Br(\R(S))$ be a
Brauer
 class of period $n$. 
If $n$ is odd, or if there exists a lift $\xi\in H^2_G(S(\C),\Z/2)$ of $\frac{n}{2}\alpha\in \Br(S)[2]$ in (\ref{BrliftS}) such that 
 $([\xi]_1)|_{\Theta}=0\in H^1(\Theta,\Z/2)$, then $\ind(\alpha)=\per(\alpha)$. Otherwise, $\ind(\alpha)=2\per(\alpha)$.
\end{thm}

When $H^2(S,\cO_S)=0$, the condition of Theorem \ref{pinr} is purely topological and may be checked in practice by concrete computations. 
In \S\ref{EnriquesR}, we illustrate this in the case of Enriques surfaces. To state this result, we recall that the real locus of an Enriques surface $S$ over $\R$ has a canonical decomposition $S(\R)=S_1\sqcup S_2$ as a disjoint union of two open and closed subsets, called the halves of $S$ (see \cite[\S 1.3]{Halves}).

\begin{thm}
\label{thEnr}
Let $S$ be an Enriques surface over $\R$. The following are equivalent:
\begin{enumerate}[(i)]
\item There exists a class $\alpha\in\Br(S)\subset\Br(\R(S))$ such that $\ind(\alpha)\neq\per(\alpha)$.
\item The manifold $S(\R)$ is not orientable and, if exactly one of the halves of $S$ is nonempty, $S(\R)$ has an odd number of connected components with odd Euler characteristic.
\end{enumerate}
\end{thm}
The possible real loci of Enriques surfaces have been classified  by Degtyarev, Itenberg and Kharlamov \cite[Appendix C]{DIK}. Many satisfy (ii), and many do not.

\subsection{Situations where period and index differ}

If $S$ is a connected smooth projective surface over $\R$, it has been known for a long time that the period and the index of $\alpha\in\Br(\R(S))$ may not coincide. For instance, Albert  has shown that the biquaternion class $\alpha=(x,x)+(y,xy)\in \Br(\R(x,y))$ has period $2$ and index $4$ in one of the first examples of Brauer classes for which period and index differ \cite[Theorem 2]{Albert} (see also \cite[VI, Example 1.11]{Lam}). In these examples, the difference between period and index may be explained by an analysis of the ramification of $\alpha$ on $S$. We refer to \cite{CTram} for a general discussion
of the obstructions to the equality of period and index induced by the ramification.

That an unramified Brauer class $\alpha\in\Br(S)$ may have  different period and index, as in some of the examples of Theorem \ref{thEnr}, is new.
 The obstruction, described in Theorem \ref{pinr}, has an obvious topological flavour. Since the image of the morphism $\Pic(S)/2\to H^2_G(S(\C),\Z/2)$ in (\ref{BrliftS}) is controlled by Krasnov's real Lefschetz $(1,1)$ theorem \cite[Proposition 2.8]{BWI}, this obstruction also depends on the Hodge theory of the surface $S$. For this reason, it is reminiscent of Kresch's Hodge-theoretical obstructions to the equality of the period and the index of unramified Brauer classes on complex varieties \cite[Theorem 1]{Kresch}. Kresch's article is, to the best of our knowledge, the first to point out the influence of Hodge theory on period-index problems.

Theorem \ref{pinr} makes sense over any real closed field, such as $\bK:=\cup_n\R((t^{1/n}))$, if one replaces Betti cohomology with semi-algebraic cohomology.
The following proposition, proven in \S\ref{K3nonarch}, shows that it fails to hold in this greater generality.
This demonstrates the existence of further obstructions to the equality of period and index over general real closed fields.
\begin{prop}
\label{propK3nonarch}
There exists a K3 surface $S$ over $\bK:=\cup_n\R((t^{1/n}))$ such that $H^1(S(\bK),\Z/2)=0$, and a class $\alpha\in\Br(S)[2]$ such that $\ind(\alpha)=4$.
\end{prop}

\subsection{Relation with conjectures of Lang and Pfister}
\label{parLang}

Recall that a field $K$ is said to be $C_i$ if every degree $d$ hypersurface in $\bP^N_K$ with $d^i\leq N$ has a $K$-point. The main example of such fields is:

\begin{thm}[Tsen-Lang \cite{LangCi}] 
\label{TsenLang}
The function field of an integral variety $X$ of dimension $i$ over an algebraically closed field is $C_i$.
\end{thm}

Lang \cite[p.379]{LangR} has conjectured a real analogue of this theorem.

\begin{conj}[Lang]
\label{conjLang}
The function field of an integral variety $X$ of dimension $i$ over $\R$ such that $X(\R)=\varnothing$ is $C_i$.
\end{conj}

Very little is known about Conjecture \ref{conjLang}. The case $i=1$ and $d=2$ of quadrics over function fields of curves is a classical result of Witt \cite[Satz 22]{Witt}, and Lang has shown in \cite[Corollary p.390]{LangR} that Conjecture \ref{conjLang} holds for odd degrees $d$, as the proof of Theorem \ref{TsenLang} may be adapted in this case. 
As a consequence of Theorem \ref{pisanspoint}, we give further evidence for Conjecture \ref{conjLang} by solving it for $i=2$ and $d=2$ :

\begin{thm}
\label{quadrique}
Let $S$ be an integral surface over $\R$ such that $S(\R)=\varnothing$. 
Then all quadratic forms of rank $\geq 5$ over $\R(S)$ are isotropic. 
\end{thm}

Recall that a field is said to be real if it may be ordered (as a field).
The $u$-invariant $u(K)$ of a non-real field $K$ is defined 
as the maximal rank of an anisotropic quadratic form over $K$ (see \cite[Chapter 8]{Pfisterbook}, \cite[Chapter XI, \S 6]{Lam}). 
Theorem \ref{quadrique} asserts that the $u$-invariant of the function field of a real surface without real points is at most $4$. The definition of the $u$-invariant has been generalized by Elman and Lam \cite[Definition 1.1]{uinvI} to the case of real fields, as the maximal rank of an anisotropic quadratic form over $K$
 whose signature with respect to any ordering of $K$ is trivial. In this more general setting, Pfister \cite[Conjecture 2]{Pfisterabelian} (see also \cite[XIII, Question 6.5]{Lam}) proposed that the following should hold:

\begin{conj}[Pfister]
\label{conjPfister}
If $K/\R$ has transcendence degree $i$, then $u(K)\leq 2^i$.
\end{conj}

Pfister also pointed out a link between Conjecture \ref{conjPfister} for $i=2$ and the period-index problem \cite[Proposition 9]{Pfisterabelian}.
 This allows us to apply Theorem \ref{piavecpoint} to solve the $2$-dimensional case of this conjecture.
Our result is already new for $K=\R(x,y)$.

\begin{thm}
\label{uinvariant}
Let $K$ be a field of transcendence degree $2$ over $\R$. Then $u(K)\leq 4$. If moreover $K=\R(S)$ for some  integral surface $S$ over $\R$, then $u(K)=4$.
\end{thm}

Conversely,  Theorems \ref{pisanspoint} and \ref{piavecpoint} were known to be consequences of Conjectures \ref{conjLang} and \ref{conjPfister}. Indeed, they reduce to the case of Brauer classes of period $2$ by de Jong's theorem \cite{deJong} and a norm argument, and one may apply \cite[Proposition 9]{Pfisterabelian}.

By the Amer-Brumer theorem \cite[Th\'eor\` eme 1]{Brumer}, Theorem \ref{quadrique} implies that pairs of quadratic forms of rank $\geq 5$ over the function field of a real curve with no real points have a nontrivial common zero. As a consequence, we get:

\begin{thm}
\label{dP4}
Let $C$ be an integral curve over $\R$ such that $C(\R)=\varnothing$. Then every degree $4$ del Pezzo surface over $\R(C)$ has a rational point.
\end{thm}

Conjecture \ref{conjLang} for $i=1$ and the $C_1$ conjecture of Koll\'ar and Manin combine to predict that a rationally connected variety on the function field of a real curve without real points has a rational point. Theorem \ref{dP4} solves the first open case of this problem. It was also known to follow from Conjecture \ref{conjLang} \cite[Theorem 3]{LangCi}.

By \cite[Proposition 8.3]{BWII}, when $C$ is the anisotropic conic over $\R$, Theorem \ref{dP4} is equivalent to the validity of the real integral Hodge conjecture \cite[Definition 2.2]{BWI} for $1$-cycles on degree $4$ del Pezzo fibrations over $C$. This gives further evidence for its validity for $1$-cycles on rationally connected varieties over $\R$ \cite[Question 2.16]{BWI}.

By work of Merkurjev \cite[Theorem 4]{Merkurjev}, Theorem \ref{quadrique} does not generalize to arbitrary fields of cohomological dimension $2$. Similarly, Colliot-Th\'el\`ene and Madore \cite[Th\'eor\`eme 1.2]{CTMadore} have shown that Theorem \ref{dP4} does not hold for all fields of cohomological dimension $1$. This explains that the proofs of Theorems \ref{quadrique} and \ref{dP4} use in an essential way the geometric nature of function fields of real varieties.

  Lang and Pfister have given slightly more general formulations of Conjectures \ref{conjLang} and \ref{conjPfister}, where the field $\R$ of real numbers is replaced with an arbitrary real closed field.  Our method of proof, relying in an essential way on infinitesimal methods in Hodge theory, does not apply in this more general setting (although one may use the Tarski-Seidenberg principle to extend our main theorems to the case of archimedean real closed fields: real closed subfields of $\R$).

\subsection{Strategy of the proof}
\label{strategy}
We do not know how to adapt the existing proofs of de Jong's theorem (\cite{deJong}, \cite[\S 4.2.2]{Lieblichtwisted}) to prove our main results. Instead, we use a new approach to period-index problems, based on Hodge theory.

To explain its principle, let us outline the proofs of Theorems \ref{pisanspoint}, \ref{piavecpoint} and of the first half of Theorem \ref{pinr}, for a period $2$ class $\alpha\in\Br(\R(S))$ in the function field of a connected smooth projective surface $S$ over $\R$. We wish to show that $\ind(\alpha)=2$.

In order to do so, we construct carefully (in \S\S\ref{consdouble}--\ref{NLlocus}--\ref{topoRdouble}) a ramified double cover $p:T\to S$, and try to prove that $\alpha_{\R(T)}=0\in\Br(\R(T))[2]$. As a first step, we show in \S\ref{ramification} that $\alpha_{\R(T)}$ is unramified, i.e. belongs to $\Br(T)[2]\subset \Br(\R(T))[2]$. To analyze $\Br(T)[2]$, we make use of the exact sequence (\ref{Brdeco}):
  \begin{equation}\label{BrdecoT}
0\to H^2_G(T(\C),\Z(1))/\langle\Pic(T),2\rangle\to\Br(T)[2]\to H^3_G(T(\C),\Z(1))[2]\to 0.
\end{equation}
We prove in \S\ref{killtau} that $\alpha_{\R(T)}$ lifts in (\ref{BrdecoT}) to a class $\beta\in H^2_G(T(\C),\Z(1))$.
By (\ref{BrdecoT}), it remains to show that $\beta\in\langle\Pic(T),2\rangle$. At this point, there is no reason why it should be true. 

The idea to achieve it is to let $T$ vary in moduli. For some values of the parameter corresponding to Noether-Lefschetz loci, the surface $T$ will carry extra algebraic cycles, making it more likely that $\beta\in\langle\Pic(T),2\rangle$. To conclude, we need an abundance result for Noether-Lefschetz loci that will allow us to pick a surface $T$ for which one has indeed $\beta\in\langle\Pic(T),2\rangle$, hence $\alpha_{\R(T)}=0$.

Over $\C$, an infinitesimal criterion for the abundance of Noether-Lefschetz loci in a family of surfaces has been discovered by Green \cite[\S 5]{CHM} (see \cite[\S 17.3.4]{Voisin}). This criterion has been adapted to the real setting in \cite[\S 7.2]{BWII} and \cite[\S 1]{NL3}. In \S\S\ref{verifGreen}--\ref{biratchoice}, we verify the hypothesis of the real analogue of Green's infinitesimal criterion for some families of ramified double covers of surfaces, thus completing the proof.

\vspace{1em}

Since this proof is long and technical, we first illustrate our approach in a simplified situation in Section \ref{secdJ}, by giving a proof of de Jong's Theorem \ref{pidJ} in the unramified complex case (Theorem \ref{dJunr}). In this setting,
Green's infinitesimal criterion has been verified by Voisin \cite{Voisin} in a generality sufficient for the argument.

There are two additional reasons to include Section \ref{secdJ}. First, since de Jong's Theorem \ref{pidJ} may be reduced to characteristic $0$ by \cite[\S 4.1.2]{Lieblichtwisted}, to $\C$ by the Lefschetz principle, and to the unramified case by
\cite[\S 7]{deJong}, it yields an alternative proof of this theorem. Second, our method provides new information about Theorem \ref{pidJ} in the unramified case: we obtain a density result for covers splitting a fixed unramified Brauer class on a complex surface (see Proposition \ref{propprecise}).

 The proofs of our main theorems are significantly more involved than that of Theorem \ref{dJunr} because one has to take into account the ramification and the topology of the real locus, and because no real analogue of Voisin's theorem is available. 
Although the analysis of the topology of the real locus plays obviously no role in the proof of Theorem \ref{pisanspoint}, it is very important for the proof of Theorem \ref{piavecpoint}.
Finally, we cannot use de Jong's trick \cite[\S 7]{deJong} to reduce to the unramified case. Indeed, in the process, the base field $\R$ would be replaced with a non-archimedean real closed field, where Hodge-theoretic arguments do not apply.

\subsection{Structure of the paper}
\label{structure}
As explained above, Section \ref{secdJ} is devoted to implementing our strategy in the simplified setting of unramified Brauer classes on complex surfaces. Section \ref{generalities} then gathers generalities concerning the cohomology of real algebraic varieties that are used throughout the text.

The proof of Theorems \ref{pisanspoint}, \ref{piavecpoint} and of the first half of Theorem \ref{pinr}, that has been sketched in \S\ref{strategy}, covers Sections \ref{secdouble}--\ref{secproof}. The argument itself, building on the material developed in the previous sections, can be found in \S\ref{perind2} for classes of period $2$, and in \S\ref{perind} in general. In \S\ref{paru}, it is explained why Theorem \ref{uinvariant}, hence also Theorems \ref{quadrique} and \ref{dP4}, follow from these results.

The second half of Theorem \ref{pinr}, that is the description of an obstruction to the equality of period and index, is proven in Section \ref{secpush}. The argument relies on a topological analysis of ramified covers of smooth projective surfaces over $\R$.

Finally, Section \ref{secex} illustrates our results with examples. In \S\ref{EnriquesR}, we study unramified Brauer classes on real Enriques surfaces and prove 
Theorem \ref{thEnr}. In \S\ref{K3nonarch},  we exhibit a K3 surface over a non-archimedean real closed field for which Theorem \ref{pinr} fails, thus proving Proposition \ref{propK3nonarch}.

\vspace{0.5em}

{\it Acknowledgements.} I would like to thank to Olivier Wittenberg for many useful discussions on related topics. The author is partly supported by ANR grant ANR-15-CE40-0002-01.

\section{Unramified Brauer classes on complex surfaces}
\label{secdJ}

In this section, we illustrate our method by proving de Jong's Theorem \ref{pidJ} for unramified classes on complex surfaces. As we have already explained in \S\ref{strategy}, the full statement of de Jong's theorem may be reduced to this case.

\begin{thm}[de Jong]
\label{dJunr}
Let $S$ be a connected smooth projective surface over $\C$. If $\alpha\in\Br(S)\subset\Br(\C(S))$, then $\ind(\alpha)=\per(\alpha)$.
\end{thm}

  Let $A$ be a very ample line bundle on $S$, chosen sufficiently positive so that $A^2>K_SA$. Introduce the threefold  $\kX:=\mathbb{P}^1_{\C}\times S$, and consider the very ample line bundle $H:=p_1^*\cO_{\mathbb{P}^1}(1)\otimes p_2^*A$ on $\kX$. Our hypothesis on $A$ implies that $H^2 K_\kX<0$. 

Let $n$ be the period of $\alpha$. If $d\geq 1$ is a fixed integer, we will denote by $B\subset|dnH|$ the Zariski open locus parametrizing smooth surfaces and by $\cT\to B$ the universal family over $B$. 
  We obtain the following  more precise result:

\begin{prop}\label{propprecise}
If $d\gg 0$ is big enough, the $b\in B(\C)$ such that $\alpha|_{\cT_b}=0\in\Br(\cT_b)$ are dense in $B(\C)$ for the euclidean topology.
\end{prop}

\begin{proof}[Proof of Theorem \ref{dJunr}]
Apply Proposition \ref{propprecise} to two big enough consecutive integers $d$ and $d+1$. We obtain two smooth surfaces $T\in |dnH|$ and $T'\in |(d+1)nH|$ on which $\alpha$ vanishes. These surfaces are respectively of degree $dn$ and $(d+1)n$ over $S$, and we deduce that $\ind(\alpha)\mid dn$ and $\ind(\alpha)\mid(d+1)n$, hence that $\ind(\alpha)\mid n$. Since $n=\per(\alpha)\mid\ind(\alpha)$ is automatic, $\ind(\alpha)=n$, as wanted.
\end{proof}

We now fix $d\geq 1$, and turn to the proof of Proposition \ref{propprecise}. For any variety $X$
over $\C$, the Kummer exact sequence $1\to\bmu_n\to\G_m\xrightarrow{n}\G_m\to 1$ of \'etale sheaves on $X$ and the comparison isomorphism $H^2_{\et}(X,\bmu_n)\xrightarrow{\sim} H^2(X(\C),\Z/n)$ between \'etale and Betti cohomology \cite[Th\'eor\`eme 4.1]{Artincomparaison} induce a short exact sequence:
\begin{equation}
\label{BrliftC}
0\to \Pic(X)/n\to H^2(X(\C),\Z/n)\to \Br(X)[n]\to 0.
\end{equation}
Comparing (\ref{BrliftC}) with the long exact sequence associated to the short exact sequence
 $0\to\Z\xrightarrow{n}\Z\to\Z/n\to 0$ of sheaves on $X(\C)$, we get an exact sequence (see \cite[\S 2]{BeauvilleEnr}): 
\begin{equation}
\label{BrauerC}
0\to H^2(X(\C),\Z)/\langle n,\Pic(X)\rangle\to \Br(X)[n]\to H^3(X(\C),\Z)[n]\to 0.
\end{equation}

In the two following lemmas, we fix any  point $b\in B(\C)$. Let $T:=\cT_b$ be the associated smooth surface, and
 let $p:T\to S$ be the projection. 

\begin{lem}\label{lem1}
The image $\tau\in H^3(T(\C),\Z)[n]$ of $p^*\alpha\in\Br(T)[n]$  by (\ref{BrauerC}) vanishes.
\end{lem}

\begin{proof}
We consider the commutative diagram:
\begin{equation*}
\xymatrix
@R=0.4cm 
@C=0.5cm 
{
H^3(S(\C),\Z)\ar[d]_{p^*}\ar[r]
&  H^3(\kX(\C),\Z)\ar[ld]_(.57){i^*}\ar^{\cupp \cl_\C(T)}[d]  \\
H^3(T(\C),\Z)\ar^{i_*}_{\sim}[r]
& H^5(\kX(\C),\Z),
}
\end{equation*}
where $i:T\to \kX$ is the inclusion. The composition $i_*i^*$ is the cup product by the cycle class $\cl_\C(T)\in H^2(\kX(\C),\Z)$ of $T$, by the projection formula. Since $\cl_\C(T)=ndH$ is divisible by $n$ in $H^2(\kX(\C),\Z)$, we deduce that the image $\tau_S\in H^3(S(\C),\Z)[n]$ of $\alpha$ by (\ref{BrauerC}) vanishes in $H^5(\kX(\C),\Z)$ in the diagram above. But $i_*$ is an isomorphism by the weak Lefschetz theorem, so that $\tau=p^*\tau_S=0\in H^3(T(\C),\Z)[n]$.
\end{proof}

By the exact sequence (\ref{BrauerC}), $p^*\alpha\in\Br(T)[n]$ then lifts to a class $\beta\in H^2(T(\C),\Z)$.

\begin{lem}\label{lem2}
There exists $\gamma\in H^2(T(\C),\Z)$ such that $p_*(\beta-n\gamma)\s=0\in H^2(S(\C),\Z)$.
\end{lem}

\begin{proof}
Let $\talpha\in H^2(S(\C),\Z/n)$ be a lift of $\alpha$ in (\ref{BrliftC}).
Then the image of $p_*\beta\in H^2(S(\C),\Z)$ in $H^2(S(\C),\Z/n)$ is $p_*p^*\talpha=dn\talpha=0\in H^2(S(\C),\Z/n)$. The short exact sequence $H^2(S(\C),\Z)\xrightarrow{n}H^2(S(\C),\Z)\to H^2(S(\C),\Z/n)$ shows that there exists $\vep\in H^2(S(\C),\Z)$ such that $n\vep=p_*\beta$.

The composition $H^2(T(\C),\Z)\xrightarrow{i_*}H^4(\kX(\C),\Z)\to H^2(S(\C),\Z)$ of push-forward morphisms is surjective because $i_*$ is surjective by the weak Lefschetz theorem and because so is $H^4(\kX(\C),\Z)\to H^2(S(\C),\Z)$ by computation of the cohomology of $\kX(\C)=\mathbb{P}^1(\C)\times S(\C)$. It follows that there exists $\gamma\in H^2(T(\C),\Z)$ such that $p_*\gamma=\vep$. Then $p_*(\beta-n\gamma)=p_*\beta-n\vep=0\in H^2(S(\C),\Z)$, as wanted.
\end{proof}

  If $\Lambda\subset B(\C)$ is a contractible open set, and $b,x\in\Lambda$, 
Ehresmann's theorem
allows us to identify  canonically
$H^2(\cT_b(\C),\R)$ and $H^2(\cT_x(\C),\R)$.
 We will use the following difficult theorem of Voisin:
 
\begin{thm}[Voisin]
\label{Voisin}
Suppose that $d \gg 0$. Then there exists a nonempty Zariski open subset $V\subset B$ that satisfies the following property. 

If $\Lambda\subset V(\C)$ is a contractible neighbourhood of $b\in V(\C)$, there exists a nonempty open cone $\Omega$ in $H^2(\cT_b(\C),\R)_S:=\Ker[H^2(\cT_b(\C),\R)\xrightarrow{p_*}H^2(S(\C),\R)]$ with the property that for every $\nu\in \Omega$, there exists $x\in\Lambda$ such that $\nu\in H^2(\cT_x(\C),\R)$ is of type (1,1) in the Hodge decomposition of $\cT_x$.
\end{thm}

\begin{proof}
It is a particular case of the main results of \cite{Voisin} (whose notation is slightly different: in \cite{Voisin}, $\kX$, $S$ and $T$ are denoted by $X$, $\Sigma$ and $S$). Let us be more precise.

The properties required at the beginning of \cite[\S 3]{Voisin} are satisfied: $\kX$ admits a morphism $\kX\to S$ to a surface with rational generic fiber, and $H^2K_\kX<0$ by our choice of $H$. In this situation, and if $d\gg 0$, we may apply \cite[Proposition 9]{Voisin} for $n=dl$. This shows that there exists a nonempty Zariski open subset $V\subset B$ such that all surfaces $\cT_b$ for $b\in V(\C)$ satisfy the hypothesis of \cite[Proposition 8]{Voisin}.

For any such surface $\cT_b$, it is possible to run the proof of \cite[Proposition 8]{Voisin}, and the existence of an open cone $\Omega\subset H^2(\cT_b(\C),\R)_S$ satisfying the required property is an intermediate step in this proof.
\end{proof}

The conclusion of \cite[Proposition 8]{Voisin} is the validity of the integral Hodge conjecture for $1$-cycles on $\kX$, a statement that is trivial in our setting. We use here that the proof of \cite[Proposition 8]{Voisin} contains much more  information.

It is now possible to conclude.

\begin{proof}[Proof of Proposition \ref{propprecise}]
Fix $d\gg 0$ big enough so that Theorem \ref{Voisin} applies, and let $W\subset B(\C)$ be a nonempty open subset. Since $W$ is Zariski dense in $B$, it meets the Zariski open subset $V\subset B$ provided by Theorem \ref{Voisin}. Choose a point $b\in V(\C)\cap W$ and let $\Lambda\subset V(\C)\cap W\subset B(\C)$ be a contractible neighbourhood of $b$ in $V(\C)\cap W$.

 Applying Lemmas \ref{lem1} and \ref{lem2} to $T=\cT_b$ provides two classes $\beta,\gamma\in H^2(T(\C),\Z)$ such that $\beta-n\gamma\in H^2(T(\C),\Z)_S:=\Ker[H^2(T(\C),\Z)\xrightarrow{p_*}H^2(S(\C),\Z)]$. Since the image of $H^2(T(\C),\Z)_S$ in 
$H^2(T(\C),\R)_S$ is a lattice, it is possible to find a class $\delta\in H^2(T(\C),\Z)_S$ such that the image of $\beta-n\gamma-n\delta$ in $H^2(T(\C),\R)_S$ belongs to the nonempty open cone $\Omega$ provided by Theorem \ref{Voisin}.
 
Theorem \ref{Voisin} shows that there exists $x\in\Lambda$ such that $\beta-n\gamma-n\delta$ is of type $(1,1)$ in the Hodge decomposition of $H^2(\cT_x(\C),\R)$. By the Lefschetz theorem on $(1,1)$ classes, $\beta=n\gamma+n\delta+\cl_{\C}(\varphi)$ for some $\varphi\in\Pic(\cT_x)$, where $\cl_{\C}$ is the cycle class map. The exact sequence (\ref{BrauerC}) then shows that $\alpha|_{\cT_x}=0\in \Br(\cT_x)$, as wanted.
\end{proof}

\section{The cohomology of real algebraic varieties}
\label{generalities}

We collect here general facts that will be used in the remainder of the text.

\subsection{Real varieties and their cohomology}
Let $\R$ be the field of real numbers and $\C=\R[\sqrt{-1}]$ be the field of complex numbers.
Define $G:=\Gal(\C/\R)\simeq \Z/2\Z$, generated by the complex conjugation $\sigma\in G$.
A variety $X$ over $\R$ is a separated scheme of finite type over $\R$. The set $X(\C)$ of complex points of $X$, endowed with the euclidean topology, carries a continuous action of $G$ whose fixed locus is the set $X(\R)$ of real points of~$X$.
 
Let $Y\subset X(\C)$ be a locally closed $G$-invariant subset (such as $X(\C)$ or $X(\R)$). For $i\geq 0$, we set $H^i(Y):=H^i(Y,\Z/2)$. We denote by $D^+(Y)$ (resp. $D^+_G(Y)$) the bounded below derived category of sheaves (resp. of $G$-equivariant sheaves) of abelian groups on $Y$.
  If $\sF$ is a $G$-equivariant sheaf of abelian groups on $Y$, we let $H^q_G(Y,\sF)$ be its equivariant cohomology groups. If $M$ is a $G$-module, we still denote by $M$ the associated constant $G$-equivariant sheaf on $Y$, and refer to $H^q_G(Y,M)$ as an equivariant Betti cohomology group. 
The $G$-module $\Z(j):=(\sqrt{-1})^j\Z\subset \C$ only depends on the parity of $j\in \Z$ (this is the convention of \cite{BWI}; it differs from the one in \cite{NL3} where $\Z(j)=(2\pi\sqrt{-1})^j\Z\subset \C$, but this should cause no confusion). For $\sF$ and $M$ as above, we define $\sF(j):=\sF\otimes_{\Z} \Z(j)$ and $M(j):=M\otimes_{\Z} \Z(j)$.

 We will use extensively properties of equivariant Betti cohomology, for which we will refer to \cite[\S 1]{BWI} (see also \cite{vanhamelthese}). If $X$ and $X'$ are smooth equidimensional varieties of dimensions $d$ and $d'$ over $\R$ and $f:X'\to X$ is a proper morphism, we will make use of the push-forward morphism:
\begin{equation}
\label{pf}
f_*:H^{k+2(d'-d)}_G(X'(\C),M(-k))\to H^k_G(X(\C),M)
\end{equation}
 defined in \cite[(1.22)]{BWI} for any $G$-module $M$ and any $k\in\Z$.
We will also consider, for a smooth variety $X$ over $\R$, the cycle class map
$\cl_\C:\CH^k(X_\C)\to H^{2k}(X(\C),\Z(k))$
in Betti cohomology, Krasnov's cycle class map $\cl:\CH^k(X)\to H^{2k}_G(X(\C),\Z(k))$ in equivariant Betti cohomology (\cite[\S 2.1]{krasnovcharacteristicclasses}, \cite[(1.55)]{BWI}), and the cycle class map $\cl_{\R}:\CH^k(X)\to H^k(X(\R))$ defined by Borel and Haefliger (\cite[\S 5]{BH}, \cite[(1.56)]{BWI}).

\subsection{Finite \'etale double covers}
\label{doublesec}
Let $X$ be a variety over $\R$. To a finite \'etale cover $p:\tX\to X$ of degree $2$, one can associate a 
$G$-equivariant sheaf $\sL$ on $X(\C)$, that is locally constant with stalks isomorphic to $\Z$ as a non-equivariant sheaf, and that fits in
natural exact sequences :
\begin{equation}
\label{sesdoubleZ}
0\to\Z\to p_*\Z\to \sL\to 0\hspace{1em}\text{  and  }\hspace{1em} 0\to\sL\to p_*\Z\to\Z\to 0.
\end{equation}
Reducing any of the exact sequences (\ref{sesdoubleZ}) modulo $2$ yields an exact sequence:
\begin{equation}
0\to\Z/2\to p_*\Z/2\to\Z/2\to 0.
\label{sesdouble2}
\end{equation}
Let $e^{\sL}_\Z\in H^1_G(X(\C),\sL)$ be the extension class of the exact sequences (\ref{sesdoubleZ}), and let
 $e^{\sL}_{\Z/2}\in H^1_G(X(\C),\Z/2)$ its reduction modulo $2$, that is the extension class of (\ref{sesdouble2}). 
The boundary maps
 of long exact sequences of $G$-equivariant cohomology induced by (\ref{sesdoubleZ}) or  (\ref{sesdouble2}) are given by cup-products by $e^{\sL}_{\Z}$ or $e^{\sL}_{\Z/2}$. Multiplication by $2$ on $\sL/4$ gives another short exact sequence:
\begin{equation}
0\to\Z/2\to \sL/4\to\Z/2\to 0.
\label{sesdouble4}
\end{equation}
of $G$-equivariant sheaves on $X(\C)$. The boundary maps  $\beta_{\sL}$ of a long exact sequence of $G$-equivariant cohomology induced by (\ref{sesdouble4}), called twisted Bockstein maps, are the sum of the usual Bockstein $\beta_\Z$ and of the cup-product with $e^{\sL}_{\Z/2}$ (see \cite{twistedBockstein}, that applies here because $G$-equivariant cohomology may be viewed as non-equivariant cohomology using the Borel construction).

We note that it is possible to recover $e^{\sL}_{\Z/2}\in H^1_G(X(\C),\Z/2)$ from $\sL$ as the image of $1$ by the boundary map of (\ref{sesdouble4}) and $p:\tX\to X$ from $e^{\sL}_{\Z/2}$ using the comparison isomorphism
$H^1_G(X(\C),\Z/2)\simeq H^1_{\et}(X,\Z/2)$ between equivariant Betti cohomology and \'etale cohomology \cite[Corollary 15.3.1]{Scheiderer}. The data of $p:\tX\to X$, $\sL$ or  $e^{\sL}_{\Z/2}$ are thus equivalent.

Let us spell out the particular case where $p:X_\C\to X$ is the morphism given by extension of scalars. In this case, one has $\sL=\Z(1)$, and the extension class $e_\Z^{\Z(1)}\in H^1_G(X(\C),\Z(1))$ (resp. $e_{\Z/2}^{\Z(1)}\in H^1_G(X(\C),\Z/2))$ is induced by the non-zero class $\omega_{\Z}\in H^1(G,\Z(1))\simeq\Z/2$ (resp. the non-zero class $\omega_{\Z/2}\in H^1(G,\Z/2)\simeq\Z/2$), see \cite[\S 1.1.2]{BWI}. When this causes no confusion, we denote any of these classes by $\omega$. If $\sF$ is a $G$-equivariant sheaf on a $G$-invariant locally closed subset $Y\subset X(\C)$, tensoring (\ref{sesdoubleZ}) by $\sF$ and taking $G$-equivariant cohomology yields the so-called real-complex exact sequences \cite[(1.6), (1.7)]{BWI}:
\begin{equation}
\label{rc}
H^k_G(Y,\sF(j+1))\to H^k(Y,\sF)\xrightarrow{N_{\C/\R}} H^k_G(Y,\sF(j))\xrightarrow{\cupp\omega} H^{k+1}_G(Y,\sF(j+1)),
\end{equation}
where the middle arrow $N_{\C/\R}$ will be referred to as the norm map.

\subsection{Restriction to the real locus}
\label{subsecrest}
Let $X$ be a variety over $\R$, and $Y\subset X(\R)$ be a locally closed subset.
If $\sF$ is a $G$-equivariant sheaf on $Y$, we consider the composition  of derived functors $\RR\Gamma_G(Y,\sF)=\RR\Gamma(Y,\RHom_{G}(\Z,\sF))$ yielding the first spectral sequence of equivariant cohomology \cite[(4.4.1)]{Tohoku}.
To compute $\RHom_{G}(\Z,\sF)$, notice that the free resolution \cite[Chapter I (6.3)]{Brown} of the $\Z[G]$-module $\Z$  yields a left resolution $\cK_\bullet=[\dots\to\Z[G]\xrightarrow{1-\sigma}\Z[G]\xrightarrow{1+\sigma}\Z[G]\xrightarrow{1-\sigma}\Z[G]\to 0]$ of the $G$-equivariant sheaf $\Z$ on $Y$, and choose an injective resolution $\cI^\bullet$ of $\sF$. Since $\RHom_{G}(\Z,\sF)\in D^+(Y)$ is represented by the complex $\sHom_{G}(\Z,\cI^\bullet)$, it is also represented by $\sHom_{G}(\cK_\bullet,\sF)$ as both are quasi-isomorphic to the total complex of the double complex $\sHom_{G}(\cK_\bullet,\cI^\bullet)$ by \cite[Theorem 1.9.3]{KS}. Consequently,
\begin{equation}
\label{cxcohoeq}
\RHom_{G}(\Z,\sF)\simeq [0\to\sF
\xrightarrow{1-\sigma}\sF\xrightarrow{1+\sigma}\sF\xrightarrow{1-\sigma}\sF\to\dots]\in D^+(Y).
\end{equation}

For some sheaves $\sF$, the complex (\ref{cxcohoeq}) splits in $D^+(Y)$, inducing decompositions 
of the $G$-equivariant cohomology of $\sF$ studied by Krasnov \cite{krasnovequivariant} and deve\-loped in \cite[\S 1.2]{BWI}, that we now recall (in \cite{krasnovequivariant, BWI}, only the case where $Y=X(\R)$ is treated explicitely, but the arguments there go through verbatim). 

\subsubsection{$2$-torsion coefficients}
If $\sF=\Z/2$, then $\RHom_{G}(\Z,\Z/2)\simeq\bigoplus_{q\geq 0}\Z/2[-q]$, yielding
a canonical decomposition \cite[(1.26)]{BWI}
\begin{equation}
\label{can2}
H^k_G(Y,\Z/2)\xrightarrow{\sim}\bigoplus_{0\leq i\leq k}H^i(Y)
\end{equation}
 respecting cup-products \cite[(1.28)]{BWI}, for any $k\geq 0$. By \cite[\S 1.2.1]{BWI}, (\ref{can2}) may also be constructed by applying the K\"unneth decomposition to the right-hand side of the identification $H^k_G(Y,\Z/2)=H^k(Y\times BG,\Z/2)$. It follows from this description 
that the Bockstein map $\beta_\Z:H^k_G(Y,\Z/2)\to H^{k+1}_G(Y,\Z/2)$ may be computed, in terms of the canonical decompositions (\ref{can2}), by the formula
\begin{equation}
\label{decBockstein}(a_i)_{0\leq i\leq k}\mapsto (\beta_{\Z}(a_{i-1})+(k-i)a_i)_{0\leq i\leq k+1},
\end{equation}
in which $a_{-1}=a_{k+1}=0$ by convention.

 If $0\leq i\leq k$ and $\xi\in H^k_G(X(\C),\Z/2)$, let $[\xi]_i$ be the image of $\xi$ by the composition 
  \begin{equation}
  \label{ithfactor2}
H^k_G(X(\C),\Z/2)\to H^k_G(X(\R),\Z/2)\xrightarrow{\sim}\bigoplus_{0\leq i\leq k}H^i(X(\R))\to H^i(X(\R))
\end{equation}
  of the restriction to $X(\R)$, of the decomposition (\ref{can2}) and of the projection.

\subsubsection{Twisted integral coefficients} Let $\sL$ be a $G$-equivariant locally constant sheaf on $X(\C)$ with stalks isomorphic to $\Z$, associated to $p:\tX\to X$  as in \S\ref{doublesec}.  If $0\leq i\leq k$ and $\xi\in H^k_G(X(\C),\sL)$,
we define $[\xi]_i\in H^i(X(\R))$ to be the class obtained by applying (\ref{ithfactor2}) to the image of $\xi$ by the reduction modulo $2$ morphism
 $H^k_G(X(\C),\sL)\to H^k_G(X(\C),\Z/2)$.
By \cite[\S 1.2.4]{BWI}, if $\xi\in H^k_G(X(\C),\sL)$, one has:
\begin{equation}
\label{compomega}
[\xi]_i=[\xi\cupp\omega]_i\in H^i(X(\R)).
\end{equation}

Assume now that the complex conjugation acts on $\sL|_Y$ by multiplication by $(-1)^j$ for $j\in\Z$.
This condition holds exactly when $Y\subset p(\tX(\R))$ if $j$ is even (resp. when $ p(\tX(\R))\cap Y=\varnothing$ if $j$ is odd). Applying (\ref{cxcohoeq}) with $\sF=\sL$ yields $\RHom_{G}(\Z,\sL)\simeq\bigoplus_{q\geq 0}H^q(G,\Z(j))[-q]$, inducing, for any $k\geq 0$,
 a canonical decomposition \cite[(1.30)]{BWI}:
\begin{equation}
\label{canZ}
H^k_G(Y,\sL)\xrightarrow{\sim}\bigoplus_{i\geq 0} H^i(Y, \sL\otimes H^{k-i}(G,\Z(j))).
\end{equation}
The decompositions (\ref{can2}) and (\ref{canZ}) are compatible in the sense that the diagram:
\begin{equation}
\begin{aligned}
\label{cancompa}
\xymatrix
@R=0.4cm 
@C=0.4cm 
{
 H^k_G(Y,\sL)\ar[d]\ar^(0.30){\sim}[r]&  \bigoplus_i H^i(Y, \sL\otimes H^{k-i}(G,\Z(j)))\ar[r]&\bigoplus_{0\leq i \leq k\atop\congr{i}{k-j}{2}} H^i(Y)\ar^{1+\beta_{\sL}}[d] \\
 H^k_G(Y,\Z/2)\ar^{\sim}[rr]&&H^0(Y)\oplus\dots\oplus H^k(Y),
}
\end{aligned}
\end{equation}
whose left vertical arrow and right upper horizontal arrow are given by reduction modulo $2$, and where $\beta_{\sL}$ has been defined in \S\ref{doublesec},
is commutative \cite[\S 1.2.3]{BWI}.

 It follows from (\ref{cancompa}) that, when $\congru{i}{k-j}{2}$, the class $([\xi]_i)|_Y\in H^i(Y)$ may equi\-valently be computed as the image of $\xi$ by the composition
\begin{equation}
\label{ithfactorZ}
H^k_G(X(\C),\sL)\to H^k_G(Y,\sL)\xrightarrow{\sim} 
\bigoplus_{i\geq 0}
H^i(Y,\sL\otimes H^{k-i}(G,\Z(j)))\to H^i(Y)
\end{equation}
of the restriction to $Y$, of the canonical decomposition (\ref{canZ}), of the projection and of the morphism induced by $H^{k-i}(G,\Z(j))\to H^{k-i}(G,\Z/2)\simeq\Z/2$.

\subsubsection{Line bundles}
Let $X$ be a smooth variety over $\R$ and $\cM\in\Pic(X)$.
 Krasnov  \cite[Theorem 0.6]{krasnovequivariant} has shown that the cycle class maps $\cl$ and $\cl_{\R}$ are compatible:
\begin{equation}
\label{krasnovcompa}
\cl_{\R}(\cM)=[\cl(\cM)]_1\in H^1(X(\R)).
\end{equation}
The following  lemma, contained in \cite[Proof of Lemma 3.4]{BWI}, concerns $2$-torsion line bundles.
\begin{lem}
\label{algdouble}
Let $X$ be a smooth variety over $\R$ and $e\in H^1_G(X(\C),\Z/2)$. Let $\upsilon:H^1_G(X(\C),\Z/2)\to\Pic(X)$ be induced by the isomorphism \cite[Corollary 15.3.1]{Scheiderer} $H^1_G(X(\C),\Z/2)\simeq H^1_{\et}(X,\bmu_2)$  and the inclusion $\bmu_2\to\G_m$.
Then $\cl_{\R}(\upsilon(e))=[e]_1$.
\end{lem}
\begin{proof}

Consider the commutative diagram
\begin{equation}
\label{classedouble}
\begin{aligned}
\xymatrix
@C=0.5cm 
@R=0.3cm 
{
&H^1_G(X(\C),\Z/2)\ar_(.55){\upsilon}[dl] \ar[r]\ar[d]& H^1_G(X(\R),\Z/2)\s= H^0\soplus H^1(X(\R))\ar[d]\\
\Pic(X)\ar_(.35){\cl}[r]&H^2_G(X(\C),\Z(1)) \ar[r]\ar[d]&H^2_G(X(\R),\Z(1))
\s= H^1(X(\R))
\ar[d]\\
&H^2_G(X(\C),\Z/2) \ar[r]&H^2_G(X(\R),\Z/2)\s= H^0\soplus H^1\soplus H^2(X(\R)),
}
\end{aligned}
\end{equation}
whose right horizontal arrows are restrictions to $X(\R)$, whose vertical arrows are boundary maps associated to $0\to\Z(1)\xrightarrow{2}\Z(1)\to\Z/2\to 0$ and reductions modulo $2$, where we have indicated the canonical decompositions (\ref{can2}) and (\ref{canZ}). Its upper left triangle
indeed commutes because it commutes by \cite[(3.8) and \S 3.3.1]{CTtorsion} and comparison with $2$-adic cohomology. 
The image $[\cl(\upsilon(e))]_1$ of $e$ in the factor $H^1(X(\R))$ of the bottom right group of (\ref{classedouble}) equals $\cl_{\R}(\upsilon(e))$ by (\ref{krasnovcompa}). The composition of the two right vertical arrows of (\ref{classedouble}) is $\beta_{\Z(1)}$, hence given by $(a_0,a_1)\mapsto(0,a_1,\beta_{\Z}(a_1))$ in terms of the canonical decompositions by (\ref{decBockstein}) and (\ref{compomega}). That $[e]_1=\cl_{\R}(\upsilon(e))$ then follows from the commutativity of (\ref{classedouble}).
\end{proof}

\subsection{Brauer groups}
\label{parBrauer}
   Let $X$ be a smooth integral variety over $\R$. We define the Brauer group of $X$ to be $\Br(X):=H^2_{\et}(X,\G_m)$. It is a torsion group, as a subgroup of  $\Br(\R(X))$ \cite[II, Corollaire 1.10]{gdB}.
  For any integral divisor $\Gamma\subset X$,
 there is a residue map $\res_\Gamma:\Br(\R(X))\to H^1_{\et}(\R(\Gamma),\Q/\Z)$ \cite[III, Corollaire 2.2]{gdB}.

For $\alpha\in\Br(\R(X))$, there are only finitely many integral divisors $\Gamma\subset X$ 
such that $\res_\Gamma(\alpha)\neq 0$. Their union is the ramification divisor of $\alpha$. Their complement $U\subset X$ is the biggest open subset of $X$ such that $\alpha\in\Br(U)\subset \Br(\R(X))$, by purity \cite[III, Corollaire 6.2]{gdB}.
We refer to \cite[(3.9)]{CTBarbara} for an exposition of these facts.

We now describe real analogues of the exact sequences (\ref{BrliftC}) and (\ref{BrauerC}).
For $n\geq 1$, the Kummer exact sequence $1\to\boldsymbol\mu_n\to\G_m\xrightarrow{n}\G_m\to 1$ of \'etale sheaves on $X$ and the comparison isomorphism $H^2_{\et}(X,\boldsymbol\mu_n)\xrightarrow{\sim} H^2_G(X(\C),\Z/n(1))$ between \'etale cohomology and equivariant Betti cohomology \cite[Corollary 15.3.1]{Scheiderer} induce a short exact sequence:
\begin{equation}
\label{Brlift}
0\to \Pic(X)/n\to H^2_G(X(\C),\Z/n(1))\to \Br(X)[n]\to 0.
\end{equation}
Comparing (\ref{Brlift}) with the long exact sequence of equivariant cohomology associated to the short exact sequence $0\to\Z(1)\xrightarrow{n}\Z(1)\to\Z/n(1)\to 0$ of $G$-equivariant sheaves on $X(\C)$, we obtain a short exact sequence: 
\begin{equation}
\label{Brdeco}
0\to H^2_G(X(\C),\Z(1))/\langle n,\Pic(X)\rangle\to \Br(X)[n]\to H^3_G(X(\C),\Z(1))[n]\to 0,
\end{equation}
where $\cl:\Pic(X)\to H^2_G(X(\C),\Z(1))$ is Krasnov's cycle class map. The right-hand side of (\ref{Brdeco}) may be thought of as the topological part of $\Br(X)[n]$, and the left-hand side as its Hodge-theoretic part.

\section{Ramified double covers}
\label{secdouble}

Let $S$  be a connected smooth projective surface over $\R$ and let $R\subset S$ be a simple normal crossings divisor. 
Most of the proof of our main theorems
will be devoted the construction and the analysis of double covers $p:T\to S$ that are ramified over $R$. In this section, we set up the relevant notation, and basic tools.

\subsection{Construction of double covers} 
\label{consdouble}
Fix a line bundle $\cL$ on $S$.
Consider the Zariski open subset $B\subset H^0(S,\cL^{\otimes 2}(R))$ of sections $s\in H^0(S,\cL^{\otimes 2}(R))$ 
whose zero locus $D\subset S$ is smooth and intersects $R$ transversally in its smooth locus. 
The divisor $\Delta:=R\cup D\subset S$ is then a simple normal crossings divisor. 

Let $r\in H^0(S,\cO_S(R))$ be an equation of $R$.
Let $\bkX$
 be the projective bundle $\bP_S(\cO_S\oplus \cL)$, with tautological bundle $\cO_{\bkX}(1)$ and structural morphism $f:\bkX\to S$. 
 Consider the sections $v\in H^0(\bkX,\cO_{\bkX}(1))$ and $w\in H^0(\bkX,\cO_{\bkX}(1)\otimes f^*\cL^{-1})$ induced by the factors $\cO_S\hookrightarrow\cO_S\oplus\cL\simeq f_*\cO_{\bkX}(1)$ and
$\cO_S\hookrightarrow\cL^{-1}\oplus\cO_S\simeq f_*(\cO_{\bkX}(1)\otimes f^* \cL^{-1})$.
To a section $s\in B$, we associate the surface 
$\bT:=\{r v^2=sw^2\}\subset \bkX$ with projection $\bp:\bT\to S$.
 It is a finite double cover of $S$ ramified over $\Delta$. The surface $\bT$ is smooth apart from ordinary double points above the singular points of $R$.

Define $T$ to be the blow-up of $\bT$ at its singular points: it is the minimal resolution of singularities of $\bT$. We denote by $p:T\to S$ the natural morphism. 
Let $\pi:\cT\to B$ be the family of such double covers obtained by letting $s$ vary in $B$.

In Section \ref{secNL}, it will be convenient to view $\cT$ as a family of surfaces in a fixed threefold $\kX$. To do so, we notice that
the singular points of $\bT$, viewed as a subset of $\bkX$, do not depend on $s$. Indeed they are exactly the points lying above the singular locus of $R$ where $v$ vanishes. Letting $\kX\to\bkX$ 
be the blow-up of $\bkX$ at this finite number of points, $T$ identifies with the strict transform of $\bT$ in $\kX$ and the total space of the  family $\pi:\cT\to B$ may the be viewed as a hypersurface $\cT\subset B\times \kX$.

\subsection{The topology of double covers}
\label{topodouble}
We collect here a few facts that will be used in \S\ref{parwL} and \S\ref{killtau} to perform cohomological computations on the ramified double covers constructed in \S\ref{consdouble}.

Fix a double cover $p:T\to S$ as in \S\ref{consdouble}.
Let $S^*:=S\setminus\Delta$ be the locus over which $p$ is \'etale, let $S^0:=S\setminus \Sing(\Delta)$ be the locus over which $p$ is finite flat. Let $j^0:S^*\hookrightarrow S^0$ be the inclusion, and $p:T^*\to S^*$ and $p:T^0\to S^0$ be the restrictions of $p$ over $S^*$ and $S^0$. Let $\tZ$ the $G$-equivariant locally constant sheaf on $S^*(\C)$ associated to the finite \'etale double cover 
 $p:T^*\to S^*$ as in \S\ref{doublesec}.

A local computation at points $x\in S^0(\C)\setminus S^*(\C)$ shows that $j^0_*\Z=\Z$, that $j^0_*p_*\Z=p_*\Z$, that $\RR^1j_*^0p_*\Z\xrightarrow{\sim}\RR^1j_*^0\Z$, and that  $\RR^ij_*^0p_*\Z=\RR^ij_*^0\Z=0$ if $i\geq 2$. Consider the distinguished triangles in $D^+_G(S^0(\C))$:
$\RR j_*^0\tZ\to  \RR j_*^0 p_*\Z\to  \RR j_*^0 \Z\xrightarrow{}$ (obtained by applying $\RR j_*^0$ to (\ref{sesdoubleZ})) as well as $\Z\to \RR j_*^0\Z\to \RR^1 j_*^0 \Z[-1]\to$ and 
$p_*\Z\to \RR j_*^0p_*\Z\to \RR^1 j_*^0p_* \Z[-1]\to$ (obtained by applying truncation functors).
By the octahedron axiom of triangulated categories, they fit in
a commutative diagram of distinguished triangles in $D^+_G(S^0(\C))$:
\begin{equation}
\label{distinguisheddiagram}
\begin{aligned}
\xymatrix
@R=0.25cm 
@C=0.45cm 
{
\RR j_*^0\tZ\ar^{\kappa}[r]\ar@{=}[d]&p_*\Z\ar[d]\ar[r]& \Z\ar[d] \ar[r] &\\
\RR j_*^0\tZ\ar[r]&\RR j_*^0 p_*\Z\ar[r]\ar[d]&\RR j_*^0 \Z\ar[d] \ar[r] &\\
& \RR^1j_*^0p_*\Z[-1]\ar@{=}[r]\ar[d] & \RR^1j_*^0\Z[-1]\ar[d]& \\
&&&
}
\end{aligned}
\end{equation}
In the sequel, we will make use of the first row of (\ref{distinguisheddiagram}):
\begin{equation}
\label{firstrow}
\RR j_*^0\tZ\xrightarrow{\kappa} p_*\Z\to \Z\to.
\end{equation}
Applying \cite[Proposition 1.1.9]{BBD} to the two vertical distinguished triangles in (\ref{distinguisheddiagram}) shows that the morphism $p_*\Z\to\Z$ in (\ref{firstrow}) is the unique one inducing a morphism between these triangles, hence is obtained by applying $j^0_*$ to the morphism $p_*\Z\to\Z$ on $S^*(\C)$ appearing in (\ref{sesdoubleZ}). In contrast, the morphism $\kappa:\RR j_*^0\tZ\to p_*\Z$ in (\ref{firstrow}) may not be uniquely determined.
Axiom (TR3) of triangulated categories gives rise to a morphism  of distinguished triangles in $D^+_G(S^0(\C))$:
\begin{equation}
\label{Phiphi}
\begin{aligned}
\xymatrix
@R=0.3cm 
@C=0.45cm 
{
\RR j_*^0\tZ\ar^{\kappa}[d]\ar^2[r]&\RR j_*^0\tZ\ar^{\kappa}[d]\ar[r]& \RR j_*^0(\Z/2)\ar^{\kappa_2}[d]\ar[r]&\\
p_*\Z\ar^2[r]&p_*\Z\ar[r]&p_*\Z/2\ar[r]&
}
\end{aligned}
\end{equation}
Restricting (\ref{distinguisheddiagram}) and (\ref{Phiphi}) to $S^*(\C)$ shows that the morphisms $\kappa|_{S^*(\C)}: \tZ\to p_*\Z$ and $\kappa_2|_{S^*(\C)}: \Z/2\to p_*\Z/2$ are the natural ones appearing in (\ref{sesdoubleZ}) and (\ref{sesdouble2}).

\subsection{A weak Lefschetz theorem}
\label{parwL}
Recall that we used the weak Lefschetz theorem twice in the proof of Theorem \ref{dJunr}, in Lemmas \ref{lem1} and \ref{lem2}.
The goal of this paragraph is to prove Proposition \ref{weakL}, that will serve as a substitute for the weak Lefschetz theorem in the proofs of our main results.

Real variants of the weak Lefschetz theorem have been studied in \cite[\S 1.5]{BWI}. As in \emph{loc.\ cit.}, we deduce real statement from the usual (complex) statements using real-complex exact sequences. The arguments of \emph{loc.\ cit.} are given for hypersurfaces and we adapt them to the setting of double covers.

We keep the notation of \S\S\ref{consdouble}--\ref{topodouble}. Define $P:=\Sing(\Delta)=S\setminus S^0$ (it is a finite union of points), and let $E:= p^{-1}(P)=T\setminus T^0$.
The complex variety $E(\C)$ is a disjoint union of copies of $\bP^1(\C)$, one above each point of $P(\C)$.

We will make the following hypothesis on the line bundle $\cL$ on $S$ fixed in \S\ref{consdouble}:
\begin{assumption}
\label{assumample}
The line bundle $\cL(R)$ is ample.
\end{assumption}

We start by proving the complex analogue of Proposition \ref{weakL}:
\begin{prop}
\label{weakLC}
Under Assumption \ref{assumample}, the morphism
$p_*:H^i(T(\C),\Z)\to H^i(S(\C),\Z)$ is an isomorphism for $i\geq 3$, and is surjective for $i=2$.
\end{prop}

\begin{proof}
Since $\cO_S(\Delta)\simeq\cL^{\otimes 2}(2R)$ is ample, $S^*$ is affine, and so is its finite cover $T^*$. Since they are two-dimensional, $H^i(S^*(\C),\Z)=H^i(T^*(\C),\Z)=0$ for $i\geq 3$
 (combine \cite[Theorem 1]{AF} and the universal coefficient theorem).
The long exact sequence associated to (\ref{sesdoubleZ})
on $S^*(\C)$ then shows that 
$H^i(S^*(\C),\tZ)=0$ for $i\geq 3$. The long exact sequence of cohomology associated to  (\ref{firstrow}) reads:
$$H^{i}(S^*(\C),\tZ)\to H^i(T^0(\C),\Z)\to H^i(S^0(\C),\Z)\to H^{i+1}(S^*(\C),\tZ),$$
and shows that 
$p_*:H^i(T^0(\C),\Z)\to H^i(S^0(\C),\Z)$ is an isomorphism if $i\geq 3$ and surjective if $i=2$. A diagram chase in the commutative diagram:
\begin{equation*}
\xymatrix
@R=0.3cm 
@C=0.5cm  
{
\cdots\ar[r]& H^{i-2}(E(\C),\Z)\ar[r]\ar_{p_*}[d]&H^i(T(\C),\Z)\ar[d]_{p_*}\ar[r]
&  H^i(T^0(\C),\Z)\ar_{p_*}[d]\ar[r]& H^{i-1}(E(\C),\Z)\ar_{p_*}[d]  \\
\cdots\ar[r]& H^{i-4}(P(\C),\Z)\ar[r]&H^i(S(\C),\Z)\ar[r]
& H^i(S^0(\C),\Z)\ar[r]& H^{i-3}(P(\C),\Z),
}
\end{equation*}
whose rows are long exact sequences of cohomology with support and whose vertical arrows are push-forward maps,
then implies that $p_*:H^i(T(\C),\Z)\to H^i(S(\C),\Z)$ is an isomorphism for $i\geq 3$ and is surjective for $i=2$, because $p_*:H^j(E(\C),\Z)\to H^{j-2}(P(\C),\Z)$ is an isomorphism for $j\geq 1$. 
\end{proof}

Define $I\subset H^1(S(\R))$ to be the image of $p_*:H^1(T(\R))\to H^1(S(\R))$.

\begin{prop}
\label{weakL}
Under Assumption \ref{assumample},  for $i\geq 2$,
the image of the morphism $p_*:H^i_G(T(\C),\Z(i-1))\to H^i_G(S(\C),\Z(i-1))$ is $\{\xi\in H^i_G(S(\C),\Z(i-1))\ |\ [\xi]_1\in I\}$.
\end{prop}

\begin{proof}
We argue as in \cite[\S 1.5]{BWI}, by decreasing induction on $i$. If $i\geq 5$, the diagram:
\begin{equation*}
\xymatrix
@R=0.3cm 
@C=0.5cm  
{
H^i_G(T(\C),\Z(i-1))\ar_{p_*}[d]\ar^{\sim}[r]& H^i_G(T(\R),\Z(i-1))\ar^(.6){\sim}[r]
&  H^1(T(\R))\ar_{p_*}[d] \\
H^i_G(S(\C),\Z(i-1))\ar^{\sim}[r]& H^i_G(S(\R),\Z(i-1))\ar^(.6){\sim}[r]&H^1(S(\R)),
}
\end{equation*}
whose horizontal maps are restrictions to the real locus (that are isomorphisms by \cite[\S 1.1.3]{BWI}) and canonical decompositions (\ref{canZ}), and whose vertical maps are push-forwards, commutes by \cite[Proposition 1.22]{BWI} and (\ref{cancompa}), proving the proposition.

If $i\geq 2$ and the proposition is proven for $i+1$, consider the commutative diagram:
\begin{equation*}
\xymatrix
@R=0.3cm 
@C=0.27cm  
{
H^i(T(\C),\Z(i-1))\ar_{p_*}[d]\ar[r]& H^i_G(T(\C),\Z(i-1))\ar_{p_*}[d]\ar[r]
&  H^{i+1}_G(T(\C),\Z(i))\ar_{p_*}[d]\ar[r]
&  H^{i+1}(T(\C),\Z(i))\ar_{p_*}[d] \\
H^i(S(\C),\Z(i-1))\ar[r]& H^i_G(S(\C),\Z(i-1))\ar^{\hspace{0.5em}\cupp\omega}[r]&H^{i+1}_G(S(\C),\Z(i))\ar[r]&H^{i+1}(S(\C),\Z(i))
}
\end{equation*}
whose horizontal arrows are real-complex exact sequences (\ref{rc}) and whose vertical arrows are push-forward morphisms. We deduce from the induction hypothesis and Proposition \ref{weakLC} that the image of $p_*:H^i_G(T(\C),\Z(i-1))\to H^i_G(S(\C),\Z(i-1))$ is $\{\xi\in H^i_G(S(\C),\Z(i-1))\ |\ [\xi\cupp\omega]_1\in I\}$, which concludes by (\ref{compomega}).
\end{proof}

\section{The pull-back of the Brauer class on a double cover}
\label{sectopo}
In this section, we fix a 
connected smooth projective surface $S$ over $\R$, and a class
$\alpha\in \Br(\R(S))[2]$ of period $2$ with simple normal crossings ramification divisor $R\subset S$ (see \S\ref{parBrauer}). Define $U:=S\setminus R$
and $\Theta:=\{x\in U(\R)\ |\ \alpha|_x\neq 0\in\Br(\R)\}$, and fix an open and closed subset $\Psi\subset U(\R)$ containing $\Theta$.
We also fix a line bundle $\cL$ on $S$, and a double cover $p:T\to S$
as in \S\ref{consdouble}, and use the notation of \S\S\ref{consdouble}--\ref{topodouble}.

Let $T_U:=p^{-1}(U)\subset T$. We study $p^*\alpha\in\Br(T_U)[2]\subset\Br(\R(T))[2]$.
Our goal is to prove Propositions \ref{ramiprop}, \ref{lem1R} and \ref{lem2R}. The latter two are real counterparts of Lemmas \ref{lem1} and \ref{lem2}.
We will make the following hypotheses:

\begin{assumption}
\label{assumkilltop}
The morphism $p:T\to S$ satisfies that:
\begin{enumerate}[(i)]
\item The image $p(T(\R))\subset S(\R)$ is disjoint from $\Psi$.
\item The kernel of the restriction map $H^1(S(\R))\to H^1(\Psi)$ is generated by the image of the push-forward $p_*:H^1(T(\R))\to H^1(S(\R))$ and by the Borel-Haefliger classes of curves on $S$ whose real locus does not meet $\Psi$.
\item The push-forward $p_*:H^2(T(\R))\to H^2(S(\R))$ is injective.
\end{enumerate}
\end{assumption}
Since $\Delta(\R)\subset p(T(\R))$, Assumption \ref{assumkilltop} (i) implies that $\Psi$ does not meet $\Delta(\R)$, or in other words that $\Psi\subset S^*(\R)$.

\subsection{The ramification}
\label{ramification}
We first deal with the ramification of $\alpha$.

\begin{prop}
\label{ramiprop} Under Assumption \ref{assumkilltop}, $\alpha_{\R(T)}$ belongs to $\Br(T)[2]\subset\Br(\R(T))[2]$.
\end{prop}

\begin{proof}
As explained in \S\ref{parBrauer}, we have to show that the residue $\res_{\Gamma}(\alpha)$ of $\alpha$ along an integral curve $\Gamma\subset T$ vanishes. If $\Gamma$ does not lie over $R$, this is obvious. If $p(\Gamma)$ is an irreducible component of $R$, this follows from \cite[Proposition 3.3.1]{CTBarbara}.

Otherwise, $\Gamma$ is an exceptional divisor of the blow-up $T\to\bT$. At this point, we know that the ramification locus of $\alpha_{\R(T)}$ is smooth. Since the Gersten complexes (appearing in the first page of the coniveau spectral sequence \cite[Proposition 3.9]{BO}) are complexes, the class $\res_{\Gamma}(\alpha)\in H^1(\R(\Gamma),\Z/2)$ is unramified, hence belongs to $H^1_{\et}(\Gamma,\Z/2)$.
 But $\Gamma\simeq\bP^1_{\C}$ or $\Gamma\simeq \bP^1_{\R}$. In the first case, 
$H^1_{\et}(\Gamma,\Z/2)=0$ and we are done. In the second case, $H^1(\R,\Z/2)\xrightarrow{\sim} H^1_{\et}(\Gamma,\Z/2)$, and it suffices to show that $\res_{\Gamma}(\alpha)|_x=0\in H^1(\R,\Z/2)$ for some $x\in\Gamma(\R)$.

To do so, choose a curve $\Gamma'\subset T$ that meets $\Gamma$ transversally at $x$, and that intersects $T_U$. By Assumption \ref{assumkilltop} (i) and since $\Theta\subset\Psi$, one has $p^*\alpha|_y=0\in \Br(\R)$ for $y\in T_U(\R)$, hence for $y\in\Gamma'(\R)$ general. By a theorem of Witt, it follows that $p^*\alpha|_{\R(\Gamma')}\in\Br(\R(\Gamma'))=0$. Indeed, $p^*\alpha|_{\C(\Gamma')}=0$ by cohomological dimension, so that $p^*\alpha|_{\R(\Gamma')}$ is the class of a conic, and this conic is split by \cite[Satz 22]{Witt}.
As a consequence, the residue of $p^*\alpha|_{\R(\Gamma')}$ at $x$ vanishes. Since it coincides with the restriction of $\res_{\Gamma}(\alpha)$ at $x$, the proof is complete.
\end{proof}

\subsection{The topological Brauer class}
\label{killtau}
Under the hypothesis of Proposition \ref{ramiprop}, $\alpha_{\R(T)}$ belongs to the subgroup $\Br(T)[2]\subset\Br(\R(T))[2]$. In \S\ref{killtau}, we study the image $\tau\in H^3_G(T(\C),\Z(1))[2]$ of $\alpha_{\R(T)}$ by (\ref{Brdeco}).

\begin{prop}
\label{kill1}
Under Assumptions \ref{assumample} and \ref{assumkilltop}, $\tau\cupp\omega=0\in H^4_G(T(\C),\Z)$.
\end{prop}

\begin{proof}
We consider the diagram:
\begin{equation}
\label{diagomegabr}
\begin{aligned}
\xymatrix
@R=0.3cm 
@C=0.4cm 
{
\Br(T)[2]\ar[r]\ar_{p_*}[d]&H^3_G(T(\C),\Z(1))\ar[d]_{p_*}\ar^(.55){\cupp\omega}[r]
&  H^4_G(T(\C),\Z)\ar_{p_*}[d]\ar^{}[r]& H^0(T(\R))\oplus H^2(T(\R))\ar_{p_*}[d]  \\
\Br(S)[2]\ar[r]&H^3_G(S(\C),\Z(1))\ar^(.55){\cupp\omega}[r]
& H^4_G(S(\C),\Z)\ar^{}[r]& H^0(S(\R))\oplus H^2(S(\R)),
}
\end{aligned}
\end{equation}
whose vertical arrows are push-forward maps, whose left horizontal arrows stem from (\ref{Brdeco}), and whose right horizontal arrows are $\xi\mapsto([\xi]_0,[\xi]_2)$. It commutes (by \cite[Proposition 1.22]{BWI} and (\ref{cancompa}) for the right-hand square). 
Since $p_*(\alpha_{\R(T)})=p_*p^*\alpha=2\alpha=0\in \Br(\R(S))$, 
(\ref{diagomegabr}) shows that  $p_*[\tau\cupp\omega]_2=0\in H^2(S(\R))$. By Assumption \ref{assumkilltop} (iii), we see that $[\tau\cupp\omega]_2=0\in H^2(T(\R))$. By Assumption \ref{assumkilltop} (i), $p^*\alpha|_x=0\in \Br(\R)$ for every $x\in T(\R)$, implying that $[\tau\cupp\omega]_0=0\in H^0(T(\R))$. 
By \cite[Lemma 2.11 (ii)]{BWI}, the upper right horizontal arrow of (\ref{diagomegabr}) fits into an exact sequence:
\begin{equation*}
\label{analyseH4}
H^4(T(\C),\Z)\xrightarrow{N_{\C/\R}} H^4_G(T(\C),\Z)\xrightarrow{} H^0(T(\R))\oplus H^2(T(\R)),
\end{equation*}
whose  first arrow is the norm map, showing that 
$\tau\cupp\omega$ is the norm of a class in $H^4(T(\C),\Z)$.
By Proposition \ref{weakLC} applied with $i=4$, the connected components of $T(\C)$ and $S(\C)$ are in bijection. Since the real surface $S$ is connected, we deduce that $T(\C)$ is either connected or has two connected components exchanged by the complex conjugation $\sigma\in G$. In either case, the image of the norm map $H^4(T(\C),\Z)\to H^4_G(T(\C),\Z)$ is generated by the norm $N_{\C/\R}(\cl_\C(y))$ of the cycle class of any point $y\in T(\C)$. It follows that there exists $n\in\Z$ such that $\tau\cupp\omega=nN_{\C/\R}(\cl_\C(y))=n\cl(y+\sigma(y))$. In the commutative diagram:
\begin{equation*}
\xymatrix
@R=0.3cm 
{
 H^4_G(T(\C),\Z)\ar_{p_*}[d]\ar[r]&  H^4(T(\C),\Z)\ar_{p_*}[d]  \\
 H^4_G(S(\C),\Z)\ar[r]&H^4(S(\C),\Z),
}
\end{equation*}
the class $\tau\cupp\omega$ dies in $H^4(S(\C),\Z)$ because $p_*(\tau\cupp\omega)=0$. We deduce that $p_*(n\cl_\C(y+\sigma(y)))=n\cl_\C(p(y)+p(\sigma(y)))=0$. Since $p(y)+p(\sigma(y))$ is a nontrivial effective $0$-cycle, we see that $n=0$, hence that  $\tau\cupp\omega=0$, as wanted.
\end{proof}

We now find conditions under which $\tau$ vanishes.

\begin{prop}
\label{lem1R}
 Under Assumptions \ref{assumample} and \ref{assumkilltop}, and if there exists a lift $\talpha\in H^2_G(U(\C),\Z/2)$ of $\alpha$ in $(\ref{Brlift})$ such that $([\talpha]_1)|_\Psi=0\in H^1(\Psi)$, one has $\tau=0$.
\end{prop}

\begin{proof}
Consider the commutative diagram:
\begin{equation}
\label{diagpushforwardrc}
\begin{aligned}
\xymatrix
@C=0.5cm 
@R=0.3cm 
{
 &H^3_G(T(\C),\Z)\ar_{p_*}[d]\ar[r]&  H^3(T(\C),\Z)\ar^{\wr}_{p_*}[d]\ar^(0.45){N_{\C/\R}}[r]&  H^3_G(T(\C),\Z(1))\ar_{p_*}[d]  \\
 H^2_G(S(\C),\Z(1))\ar^(0.55){\cupp\omega}[r]&H^3_G(S(\C),\Z)\ar[r]&H^3(S(\C),\Z)\ar[r] &H^3_G(S(\C),\Z(1)),
}
\end{aligned}
\end{equation}
whose horizontal arrows are real-complex exact sequences (\ref{rc}), whose vertical arrows are push-forward maps, and whose middle vertical arrow is an isomorphism by Proposition \ref{weakLC}.
Since $\tau\cupp\omega=0$ by Proposition \ref{kill1}, there exists $\gamma\in H^3(T(\C),\Z)$ such that $\tau=N_{\C/\R}(\gamma)$.
 Since $p_*\tau=0$ as was explained in the proof of Proposition  \ref{kill1}, one can lift $p_*\gamma$ to a class $\delta\in H^3_G(S(\C),\Z)$. 
We claim that, up to replacing $\delta$ with $\delta-(\xi\cupp\omega)$ for some class $\xi\in H^2_G(S(\C),\Z(1))$, we may assume that $([\delta]_1)|_\Psi=0\in H^1(\Psi)$.

By Assumption \ref{assumkilltop} (ii), this claim implies that there exists $\theta\in\Pic(S)$ such that $[\delta]_1-\cl_ {\R}(\theta)$ is in the image of $p_*:H^1(T(\R))\to H^1(S(\R))$. By (\ref{compomega}) and (\ref{krasnovcompa}), one has $\cl_{\R}(\theta)=[\cl(\theta)\cupp\omega]_1$, and Proposition \ref{weakL} shows 
that $\delta=p_*\varepsilon+\cl(\theta)\cupp\omega$ for some $\varepsilon\in H^3_G(T(\C),\Z)$.
 Since the middle vertical arrow in (\ref{diagpushforwardrc}) is an  isomorphism, it follows from a diagram chase in (\ref{diagpushforwardrc}) that $\tau=0$.

It remains to prove the above claim. Represent the distinguished triangle (\ref{firstrow}) by a short exact sequence $0\to\cI_1^\bullet\to\cI_2^\bullet\to\cI_3^\bullet\to 0$ of bounded below complexes of injective $G$-equivariant sheaves on $S^0(\C)$ \cite[Proposition 6.10]{Iversen}. 
For $\sF=\Z$, $\Z[G]$ or $\Z(1)$, applying the tensor product termwise produces complexes $\cI_i^\bullet\otimes _{\Z}\sF$ of injective objects, that represent $\cI_i^\bullet\otimes_\Z^{\LL} \sF$ because $\sF$ is $\Z$-flat.
 Tensoring by the exact sequence $0\to\Z\to\Z[G]\to\Z(1)\to 0$ and applying $\RR\Gamma_G(-):=\RR\Gamma_G(S^0(\C),-)$ termise yields a commutative exact diagram of complexes of abelian groups:
\begin{equation}
\begin{aligned}
\label{squaredist}
\xymatrix
@C=0.5cm 
@R=0.25cm {
&0\ar[d]&0\ar[d]&0\ar[d]&\\
0\ar[r]&\RR\Gamma_G(\cI_1^\bullet)\ar[r]\ar[d]&\RR\Gamma_G(\cI_1^\bullet\otimes_\Z\Z[G])\ar[d]\ar[r]& \RR\Gamma_G(\cI_1^\bullet\otimes_\Z\Z(1))\ar[d] \ar[r] \ar[r]&0\\
0\ar[r]& \RR\Gamma_G(\cI_2^\bullet)\ar[d]\ar[r]&\RR\Gamma_G(\cI_2^\bullet\otimes_\Z\Z[G])\ar[d]\ar[r]&\RR\Gamma_G(\cI_2^\bullet\otimes_\Z\Z(1))\ar[d]\ar[r]&0\\
0\ar[r]&\RR\Gamma_G(\cI_3^\bullet)\ar[d]\ar[r]&\RR\Gamma_G(\cI_3^\bullet\otimes_\Z\Z[G])\ar[d]\ar[r]&\RR\Gamma_G(\cI_3^\bullet\otimes_\Z\Z(1))\ar[d]\ar[r]&0\\
&0&0&0&\\
}
\end{aligned}
\end{equation}
Taking cohomology in (\ref{squaredist}) gives an exact commutative diagram
\begin{equation}
\label{dJannsen}
\begin{aligned}
\xymatrix
@C=0.5cm 
@R=0.3cm 
{
& &H^2_G(S^0(\C),\Z(1))\ar[d]\ar^{\hspace{0.6em}\cupp\omega}[r]&H^3_G(S^0(\C),\Z)\ar^g[d]\\
&  &  H^3_G(S^*(\C),\tZ(1))\ar^{\kappa}[d] \ar^{\hspace{0.6em}\cupp\omega}[r]& H^4_G(S^*(\C),\tZ)\\
 &  H^3(T^0(\C),\Z)\ar_{p_*}[d]\ar^(0.45){N_{\C/\R}}[r]&  H^3_G(T^0(\C),\Z(1))\ar_{p_*}[d] & \\
H^3_G(S^0(\C),\Z)\ar^{g}[d]\ar[r]&H^3(S^0(\C),\Z)\ar[r] &H^3_G(S^0(\C),\Z(1))&\\
H^4_G(S^*(\C),\tZ)& &&,
}
\end{aligned}
\end{equation}
whose rows are real-complex exact sequences (\ref{rc}). In the commutative diagram:
\begin{equation*}
\begin{aligned}
\xymatrix
@R=0.3cm 
@C=0.5cm 
{
H^2_G(S^*(\C),\Z/2)\ar^{\partial}[r]\ar^{\kappa_2}[d]&H^3_G(S^*(\C),\tZ(1))\ar[d]^{\kappa}  \\
H^2_G(T^0(\C),\Z/2)\ar^{\partial}[r]&H^3_G(T^0(\C),\Z(1))
}
\end{aligned}
\end{equation*}
obtained by twisting (\ref{Phiphi}) by $\Z(1)$ and taking $G$-equivariant cohomology, the class $\zeta:=\partial(\talpha|_{S^*})\in H^3_G(S^*(\C),\tZ(1))$ satisfies $\kappa(\zeta)=\tau|_{T^0}\in H^3_G(T^0(\C),\Z(1))$.

The two classes $g(\delta|_{S^0})$ and $\zeta\cupp\omega$ in $H^4_G(S^*(\C),\tZ)$ have been constructed from $\gamma|_{T^0}\in H^3(T^0(\C),\Z)$ by a diagram chase in (\ref{dJannsen}). By  \cite[Lemma p. 268]{Jannsenletter}
applied to diagram (\ref{squaredist}), there exists $\eta^0\in H^2_G(S^0(\C),\Z(1))$ such that
\begin{equation}
\label{eqJannsen}
g(\delta|_{S^0})+\zeta\cupp\omega=g(\eta^0\cupp\omega).
\end{equation}
 By purity \cite[(1.21)]{BWI}, the restriction $H^2_G(S(\C),\Z(1))\sto H^2_G(S^0(\C),\Z(1))$ is onto; let $\eta\in H^2_G(S(\C),\Z(1))$ be such that $\eta|_{S^0}=\eta^0$.
Let $e:=e_{\Z/2}^{\tZ}\in H^1_G(S^*(\C),\Z/2)$ be as in \S\ref{doublesec}. By Lemma \ref{algdouble} and the surjectivity of the restriction map $\Pic(S)\to\Pic(S^*)$, there exists $\varphi\in\Pic(S)$ such that $\cl_{\R}(\varphi)|_{S^*(\R)}=[e]_1\in H^1(S^*(\R))$.
 After replacing $\delta$ with $\delta-((\eta+\cl(\varphi))\cupp\omega)$, equation (\ref{eqJannsen}) becomes:
\begin{equation}
\label{threeterms}
g(\delta|_{S^0})+\zeta\cupp\omega+g(\cl(\varphi|_{S^0})\cupp\omega)=0\in H^4_G(S^*(\C),\tZ).
\end{equation}

We will use (\ref{threeterms}) to compute $([\delta]_1)|_\Psi\in H^1(\Psi)$. We first calculate  $([g(\delta|_{S^0})]_1)|_\Psi$ and $([g(\cl(\varphi|_{S^0})\cupp\omega)]_1)|_\Psi$ by considering
the commutative diagram:
\begin{equation}
\label{compZtildeZ}
\begin{aligned}
\xymatrix
@C=0.5cm 
@R=0.3cm 
{
H^3_G(S^0(\C),\Z) \ar[r]\ar^g[d]&H^3_G(\Psi,\Z)\s=H^1(\Psi)\ar[d]\ar[r]&  H^3_G(\Psi,\Z/2)\s= H^0\soplus H^1\soplus H^2(\Psi)\ar[d]  \\
H^4_G(S^*(\C),\tZ)\ar[r]&H^4_G(\Psi,\tZ)\s= H^1(\Psi)\ar[r] &H^4_G(\Psi,\Z/2)\s= H^0\soplus H^1\soplus H^2(\Psi),
}
\end{aligned}
\end{equation}
whose vertical arrows are induced by (\ref{firstrow}), whose left horizontal arrows are restrictions to $\Psi$ and whose right horizontal arrows are given by reduction mo\-dulo~$2$.
The equalities in (\ref{compZtildeZ}) are the canonical decompositions (\ref{can2}) and (\ref{canZ}); in particular, the equality $H^4_G(\Psi,\tZ)= H^1(\Psi)$ follows from the fact that complex conjugation acts by $-1$ on the stalks of $\tZ|_{\Psi}$ by Assumption \ref{assumkilltop} (i).
 By (\ref{cancompa}), the upper (resp. lower) right horizontal arrow of (\ref{compZtildeZ}) is given, in terms of the canonical decompositions, by $a\mapsto (0,a,\beta_\Z(a))$ (resp. $a\mapsto (0,a,\beta_{\tZ}(a))$). The right vertical arrow of (\ref{compZtildeZ}) is a boundary map induced by (\ref{sesdouble2}) for $\sL=\tZ$, hence is given by the cup-product by $e|_\Psi\in H^1_G(\Psi,\Z/2)$. 
By Assumption \ref{assumkilltop} (i), the class $e|_x\in H^1_G(x,\Z/2)=\Z/2$ is nontrivial for every $x\in\Psi$, so that $[e]_0=1\in H^0(\Psi)$. By the cup-product formula \cite[(1.28)]{BWI} the right vertical arrow of (\ref{compZtildeZ}) is given, in the canonical decompositions, by the formula $(a,b,c)\mapsto (a,b+a\cupp [e]_1, c+b\cupp[e]_1)$. By the commutativity of (\ref{compZtildeZ}), the middle vertical arrow of (\ref{compZtildeZ}) is the identity of $H^1(\Psi)$. We deduce that $([g(\delta|_{S^0})]_1)|_\Psi=([\delta]_1)|_\Psi\in H^1(\Psi)$. Using (\ref{compomega}) and (\ref{krasnovcompa}), we also get: $([g(\cl(\varphi|_{S^0})\cupp\omega)]_1)|_\Psi=([\cl(\varphi)\cupp\omega]_1)|_\Psi=([\cl(\varphi)]_1)|_\Psi=\cl_{\R}(\varphi)|_\Psi=([e]_1)|_\Psi$.

Our next step is to calculate $([\zeta\cupp\omega]_1)|_\Psi\in H^1(\Psi)$. By (\ref{compomega}), we know that it is equal to $([\zeta]_1)|_{\Psi}$. To compute it, we consider the commutative diagram:
\begin{equation}
\label{bordtildeZ}
\begin{aligned}
\xymatrix
@C=0.5cm 
@R=0.3cm 
{
H^2_G(S^*(\C),\Z/2) \ar[r]\ar[d]& H^3_G(S^*(\C),\tZ(1))\ar[d]\\
H^2_G(\Psi,\Z/2)\s= H^0\soplus H^1\soplus H^2(\Psi)\ar[dr]\ar[r]&H^3_G(\Psi,\tZ(1))\s= H^1(\Psi)\ar[d]\\
&H^3_G(\Psi,\Z/2)\s= H^0\soplus H^1\soplus H^2(\Psi),
}
\end{aligned}
\end{equation}
 whose vertical arrows are
restrictions to $\Psi$ and reduction modulo $2$, whose horizontal arrows are boundary maps of the exact sequence $0\to \tZ(1)\xrightarrow{2} \tZ(1)\to\Z/2\to 0$, and where we have indicated the canonical decompositions (\ref{can2}) and (\ref{canZ}). The diagonal arrow is 
a boundary map induced by (\ref{sesdoubleZ}) for $\sL=\tZ(1)$, hence is the sum of the Bockstein map and of the cup-product with $e+\omega$. Since $([e+\omega]_0)|_\Psi=1+1=0\in H^0(\Psi)$ and $([e+\omega]_1)|_\Psi= ([e]_1)|_\Psi$ by \cite[\S 1.2.5]{BWI}, the formulae \cite[(1.28)]{BWI} and (\ref{decBockstein})
show that the diagonal arrow is given by $(a,b,c)\mapsto(0,b+a\cupp[e]_1,\beta_{\Z}(b)+b\cupp[e]_1)$  in the canonical decompositions. 
By (\ref{cancompa}), the lower right vertical arrow is given by $a\mapsto (0,a,\beta_{\tZ}(a))$. Applying the commutativity of (\ref{bordtildeZ}) to $\talpha|_{S^*}\in H^2_G(S^*(\C),\Z/2)$  then shows that $([\zeta]_1)|_{\Psi}=([e]_1)|_\Psi$, by our hypothesis that $([\talpha]_1)|_\Psi=0\in H^1(\Psi)$

Plugging our computations into (\ref{threeterms}) shows that:
$$0\s=\big([g(\delta|_{S^0})]_1+[\zeta\cupp\omega]_1+[g(\cl(\varphi|_{S^0})\cupp\omega)]_1\big)|_\Psi\s=([\delta]_1)|_\Psi+2([e]_1)|_\Psi\s=([\delta]_1)|_\Psi,$$
 which completes the proof of the claim, and of the proposition.
\end{proof}

\subsection{The push-forward}
\label{pushforwarddouble}
Under the hypotheses of Proposition \ref{lem1R},  $p^*\alpha\in\Br(T)[2]$ vanishes in $H^3_G(T(\C),\Z(1))[2]$ in (\ref{Brdeco}), hence lifts to a class $\beta\in H^2_G(T(\C),\Z(1))$. We now study the class $p_*\beta\in H^2_G(S(\C),\Z(1))$. 

\begin{prop}
\label{lem2R}
Under Assumptions \ref{assumample} and \ref{assumkilltop}, and if there exists a lift $\talpha\in H^2_G(U(\C),\Z/2)$ of $\alpha$ in (\ref{Brlift}) such that $([\talpha]_1)|_\Psi=0\in H^1(\Psi)$, there exist classes $\gamma\in H^2_G(T(\C),\Z(1))$ and $\theta\in\Pic(S)$ such that $p_*(\beta-2\gamma)=\cl(\theta)\in H^2_G(S(\C),\Z(1))$.
\end{prop}
\begin{proof}
Let $p:T_U\to U$ be the restriction of $p:T\to S$ to $U\subset S$.
As $\beta|_{T_{U}}$ and $p^*\talpha$ both induce the class $p^*\alpha\in\Br(\R(T))$, the exact sequence (\ref{Brdeco}) shows that there exist $\vep\in H^2_G(T_{U}(\C),\Z(1))$ and $\varphi\in \Pic(T_{U})$ such that
$$\beta|_{T_{U}}=p^*\talpha+2\vep+\cl(\varphi).$$
Since $p_*\beta\in H^2_G(S(\C),\Z(1))$ induces $p_*p^*\alpha=2\alpha=0\in\Br(S)[2]\subset\Br(\R(S))[2]$ in the exact sequence  (\ref{Brdeco}), there exist $\delta\in H^2_G(S(\C),\Z(1))$ and $\theta_1\in\Pic(S)$ such that $p_*\beta=2\delta+\cl(\theta_1)\in
H^2_G(S(\C),\Z(1))$. We deduce that 
$$2\talpha+2p_*\vep+\cl(p_*\varphi)=(p_*\beta)|_{U}=2\delta|_{U}+\cl(\theta_1)|_{U}\in H^2_G(U(\C),\Z(1)).$$
As $\cl:\Pic(U)\to H^2_G(U(\C),\Z(1))$ has torsion-free cokernel by \cite[Proposition 2.9]{BWI}, and as the restriction map $\Pic(S)\to\Pic(U)$ is surjective, there exists $\theta_2\in\Pic(S)$ such that 
$(\delta-\cl(\theta_2))|_{U}=\talpha+p_*\vep\in H^2_G(U(\C),\Z(1))$. Since $([\talpha]_1)|_{\Psi}=0$ by hypothesis and since $[p_*\vep]_1=p_*[\vep]_1=0\in H^1(S(\R))$ by  \cite[Proposition 1.22]{BWI} and Assumption \ref{assumkilltop} (i), we get $([\delta-\cl(\theta_2)]_1)|_{\Psi}=0\in H^1(\Psi).$ It then follows from  Proposition \ref{weakL} and Assumption \ref{assumkilltop} (ii), that there exists $\gamma\in H^2_G(T(\C),\Z(1))$ such that $p_*\gamma=\delta-\cl(\theta_2)\in H^2_G(S(\C),\Z(1))$. To conclude, we set $\theta:=\theta_1+2\theta_2$ and compute:
$p_*(\beta-2\gamma)=2\delta+\cl(\theta_1)-2(\delta-\cl(\theta_2))=\cl(\theta)\in H^2_G(S(\C),\Z(1))$.
\end{proof}

\section{Small Noether-Lefschetz loci for double covers}
\label{secNL}

In Section \ref{secNL}, we fix a connected smooth projective surface $S$ over $\R$,
and a simple normal crossings divisor $R\subset S$. Our goal, achieved in Proposition \ref{propverifGreen}, is to construct a Noether-Lefschetz locus of the expected dimension in a family of double covers of $S$ that are ramified over $R$, thus verifying (the real analogue of) Green's infinitesimal criterion for this family. This will serve as a substitute of Theorem \ref{Voisin} in the proof of our main theorems.

The Noether-Lefschetz locus is constructed in \S\ref{NLlocus} and Green's criterion is checked in \S\ref{verifGreen} under restrictive assumptions on $(S,R)$. In \S \ref{biratchoice}, we explain how to ensure that these assumptions hold, by carefully replacing $S$ by a birational model.

\subsection{Curves on ramified double covers}
\label{NLlocus}
Let $\cA$ and $\cN$ be line bundles on $S$ with $\cA$ very ample. Let $l\in\bN$ be even, and define  $\cL:=\cA^{\otimes l}\otimes\cN(-R)$. Then, if $l\gg 0$ (and we choose such a $l$), the following holds:
\begin{assumption}
\label{ass2}
The groups
$H^1(S, \cA^{\otimes l})$, $H^1(S,\cL)$  and $H^1(S, \cL\otimes\cN)$ vanish. 
The line bundles $\cL$ and $\cL\otimes \cN$ are very ample.
\end{assumption}
We apply the constructions of \S\ref{consdouble} with this choice of $\cL$, and use the notation defined there. 
In particular, we defined a Zariski open subset $B\subset  H^0(S,\cL^{\otimes 2}(R))$
 and  associated with each $s\in B$ a diagram of varieties:
\begin{equation*}
\xymatrix
@R=0.3cm 
{
T\ar[d]\ar_p@/_01.0pc/[dd]\ar[r]& \kX\ar[d]  \\
 \bT\ar^(.4){\bp}[d]\ar[r]& \bkX\ar^(.4){f}[ld] \\
S &
}
\end{equation*}
We now construct a particular point $s\in B$ whose associated surface $T$ contains a curve $C\subset T$, that will give rise to an interesting Noether-Lefschetz locus.

Let $(u_i)$ be a basis of $H^0(S, \cA^{\otimes l/2})$. Since $\cA$ is very ample, the $u_i$ do not vanish simultaneously, and $\{\sum_i u_i^2=0\}\subset S$ has no real points.
Let $c\in H^0(S,\cA^{\otimes l})$ be a general small deformation of $\sum_i u_i^2\in H^0(S, \cA^{\otimes l})$. Since $\cA$ is very ample, $C:=\{c=0\}\subset S$ is a smooth curve intersecting $R$ transversally in its smooth locus, and $C(\R)=\varnothing$ as this property is preserved by small deformations.

Choose a general section $g\in H^0(C,\cL)$. Since $\cL$ is very ample by Assumption \ref{ass2}, $\{g=0\}\subset C$ is reduced and disjoint from $R$. The section $r g^2\in H^0(C,\cL^{\otimes 2}(R))$ lifts to a section $s_0\in H^0(S,\cL^{\otimes 2}(R))$ because $H^1(S,\cL\otimes\cN)=0$ by Assumption \ref{ass2}. Consider the linear system $V
\subset H^0(S,\cL^{\otimes 2}(R))$ of sections of the form $a_1s_0+a_2c$ 
with $a_1\in\R$ and $a_2\in H^0(S,\cL\otimes\cN)$. Its base locus is $\{c=s_0=0\}$ because $\cL\otimes \cN$ is very ample by Assumption \ref{ass2}, hence is finite. Let us show that a general $s\in V$ belongs to $B\subset H^0(S,\cL^{\otimes 2}(R))$, i.e. that its zero locus $D\subset S$ is smooth and intersects $R$ transversally in its smooth locus. Outside of the base locus of $V$, this follows from the Bertini theorem. At a point $x$ in the base locus of $V$, this holds for any particular choice of $(a_1,a_2)$ with $a_1=0$ and $a_2|_x\neq 0$ because $C$ is smooth and intersects $R$ transversally in its smooth locus, hence for a general choice of $(a_1,a_2)$, as wanted.
We now fix such a general $s\in V$ with $a_1>0$.

Over $C\subset S$, the finite double cover $\bp:\bT\to S$ has equation $\{rv^2=a_1rg^2 w^2\}$. It follows that the strict transform of $C$ in $\bT$ splits into two components isomorphic to $C$: one where $v=\sqrt{a_1}gw$ and one where $v=-\sqrt{a_1}gw$. We choose the first of these components and still denote it by $C\subset \bT$. It does not intersect the singular locus of $\bT$ and we still denote by $C\subset T$ its strict transform. 
Since $C\subset \bT$ satisfies the equation $v=\sqrt{a_1}gw$, and since $v$ and $w$ do not vanish simultaneously on $\bkX$, $w$ does not vanish on $C$. In other words, $w$ induces an isomorphism:
\begin{equation}
\label{isov}
\begin{aligned}
\xymatrix
@R=0.3cm 
{
\cL|_C\ar^(.4){\sim}_(.4){w}[r]&\cO_{\bkX}(1)|_C}.
\end{aligned}
\end{equation}

\subsection{Verifying the hypothesis of Green's criterion} 
\label{verifGreen}

We keep the notation of  \S\ref{consdouble} and \S\ref{NLlocus}. 
In particular, $\pi:\cT\to B$ is a family of smooth surfaces in the smooth projective threefold $\kX$ over $\R$, where we now view $B$ as a real algebraic variety in the natural way. In \S\ref{NLlocus}, we have constructed a curve $C$ in the fiber $T=\cT_s$ of $\pi$ above some $s \in B(\R)$. Let $\lambda\in H^1(T, \Omega^1_T)$ be the cohomology class of $C\subset T$ in Hodge cohomology.  We want to control the image of  the composition 
\begin{equation}
\label{defpsilambda}
\phi_{\lambda}:T_{B,s}\to H^1(T,T_{T})\xrightarrow{\lambda}H^2(T,\cO_{T})
\end{equation}
of the Kodaira-Spencer map of $\pi$ at $s$ and of the contracted cup-product with $\lambda$ induced by the pairing $T_T\otimes\Omega^1_T\to\cO_T$. We will do it under additional hypotheses:

\begin{assumption}
\label{ass1}
The cup-product morphism $H^0(S,K_S)\xrightarrow{\eta} H^1(S,\cN\otimes K_S)$ is injective for some (hence for a general) $\eta\in H^1(S,\cN)$. The groups
$H^0(S,\cN\otimes K_S)$, $H^0(S,\cN^{-1}\otimes K_S(R))$ and $H^1(R,\cN|_R)$ vanish.
\end{assumption}

We start with a lemma.
\begin{lem}
\label{lemvan}
Under Assumptions \ref{ass2} and \ref{ass1}, $h^1(C, N_{C/\kX})\leq h^2(S,\cO_S)$ and $H^2(T,\cO_T(C))=0$.
\end{lem}

\begin{proof}
To prove the first statement, we use the natural exact sequence:
\begin{equation}
\label{normalCX}
0\to T_{\bkX/S}|_C\to N_{C/\bkX}\xrightarrow{f_*} N_{C/S}\to 0.
\end{equation}
One computes $T_{\bkX/S}\simeq K_{\bkX/S}^{\vee}\simeq\cO_{\bkX}(2)\otimes \cL^{-1}$ so that $T_{\bkX/S}|_C\simeq \cL|_C$ by (\ref{isov}). Since $N_{C/S}\simeq\cO_S(C)\simeq\cA^{\otimes l}$, taking cohomology in (\ref{normalCX}) yields an exact sequence:
\begin{equation}
\label{exseqsurC}
H^1(C,\cL)\to H^1(C, N_{C/\bkX})\to H^1(C,\cA^{\otimes l}|_C).
\end{equation}
Since  $H^2(S, \cN(-R))\simeq H^0(S, \cN^{\otimes -1}\otimes K_S(R))^\vee=0$ by Assumption \ref{ass1} and $H^1(S,\cL)=0$ by Assumption \ref{ass2}, one has $H^1(C,\cL)=0$. Since $H^1(S, \cA^{\otimes l})=0$ by Assumption \ref{ass2}, we see that $h^1(C,\cA^{\otimes l}|_C)\leq h^2(S,\cO_S)$, and (\ref{exseqsurC}) implies that 
$h^1(C, N_{C/\bkX})\leq h^2(S,\cO_S)$. The blow-up $\kX\to\bkX$ being an isomorphism along $C$, 
$N_{C/\kX}\simeq N_{C/\bkX}$, which concludes.

We turn to the second statement. Since the singularities of $\bT$ are rational and avoid $C$, the Leray spectral sequence for $T\to \bT$ shows that $H^2(T,\cO_T(C))\simeq H^2(\bT,\cO_{\bT}(C))$. Since $\bT$ is a Cartier divisor in $\bkX$, it is Gorenstein with dualizing sheaf $\omega_{\bT}\simeq K_{\bkX}(\bT)|_{\bT}\simeq \bp^*(K_S\otimes \cL(R))$. By Serre duality, we need to show that $H^0(\bT,\omega_{\bT}(-C))=H^0(\bT,\bp^*(K_S\otimes \cL(R))(-C))$ vanishes. Pushing forward the exact sequence $0\to \cO_{\bT}(-C)\to\cO_{\bT}\to \cO_C\to 0$ by $\bp:\bT\to S$ yields:
\begin{equation}
\label{curveindoublecover}
0\to \bp_*\cO_{\bT}(-C)\to \cO_S\oplus\cL^{\otimes -1}(-R)\xrightarrow{(1,rg)} \cO_C,
\end{equation}
where the splitting $\bp_*\cO_{\bT}\simeq \cO_S\oplus\cL^{\otimes -1}(-R)$ is induced by the involution of the double cover $\bp:\bT\to S$. Tensoring (\ref{curveindoublecover}) with $K_S\otimes \cL(R)$ and taking cohomology gives an exact sequence:
\begin{equation*}
\label{doubleinjection}
0\to H^0(\bT,\omega_{\bT}(-C))\to H^0(S, K_S\otimes \cL(R))\oplus H^0(S, K_S)\xrightarrow{(1,rg)} H^0(C,K_S\otimes \cL(R)).
\end{equation*}
We need to show that its rightmost arrow is injective. In view of the exact sequence:
$$H^0(S, \cN\otimes K_S)\to H^0(S, K_S\otimes \cL(R))\to H^0(C,K_S\otimes \cL(R))\xrightarrow{\partial} H^1(S, \cN\otimes K_S),$$
in which the first group vanishes by Assumption \ref{ass1}, we only need to prove that the composition:
\begin{equation}
\label{compodiago}
H^0(S, K_S)\xrightarrow{rg} H^0(C,K_S\otimes \cL(R))\xrightarrow{\partial} H^1(S, \cN\otimes K_S)
\end{equation}
is injective. By \cite[Chapter II, Theorem 7.1 (c)]{Bredon}, the composition (\ref{compodiago}) identifies with the cup-product with $\partial(rg)=r\cupp\partial(g)$, 
where we have denoted by $\partial$ the boundary maps of the short exact sequences $0\to\cN\to  \cL(R)\to \cL(R)|_C\to 0$ and $0\to\cN(-R)\to  \cL\to \cL|_C\to 0$.

At this point, consider the composition:
$$\psi:H^0(C,\cL)\xrightarrow{\partial} H^1(S,\cN(-R))\xrightarrow{\cupp r}H^1(S,\cN).$$
Since $H^1(S,\cL)=H^1(R,\cN|_R)=0$ by Assumptions \ref{ass2} and \ref{ass1}, $\psi$ is surjective. Since $g\in H^0(C,\cL)$ has been chosen general, $\psi(g)=r\cupp\partial(g)\in H^1(S,\cN)$ is general. By Assumption \ref{ass1}, 
the cup-product $H^0(S,K_S)\xrightarrow{r\cupp\partial(g)}H^1(S,\cN\otimes K_S)$ is injective, as we wanted.
\end{proof}

\begin{prop}
\label{propverifGreen}
Under Assumptions \ref{ass2} and \ref{ass1}, the cokernel of the morphism $\phi_{\lambda}$ defined in (\ref{defpsilambda}) is of dimension at most $h^2(S,\cO_S)$. 
\end{prop}

\begin{proof}
By \cite[Proposition 3.2.9 (i)]{Sernesi}, the Kodaira-Spencer map may be described as the composition $T_{B,s}\to H^0(T, N_{T/\kX})\to H^1(T,T_T)$ of the map classifying infi\-ni\-tesimal deformations of $T$ in $\kX$ and of the boundary of the short exact sequence:
\begin{equation*}
\label{conormalex}
0\to T_T\to T_{\kX}|_T\to N_{T/\kX}\to 0.
\end{equation*}
It then follows from \cite[Proposition 2.1]{NL3} that $\phi_{\lambda}$ coincides with the composition:
\begin{equation}
\label{bigcomposition}
T_{B,s}\to H^0(T, N_{T/\kX})\to H^0(C, N_{T/\kX}|_C)\to H^0(C, N_{C/T})\to H^2(T,\cO_T)
\end{equation}
of the map classifying infinitesimal deformations of $T$ in $\kX$, of the restriction to $C$ and of the boundary maps of the short exact sequences
\begin{equation}
\label{normalbundles}0\to N_{C/T}\to N_{C/\kX}\to N_{T/\kX}|_C\to 0,
\end{equation}
\begin{equation}
\label{eqdeCdansT}0\to \cO_T\to \cO_T(C)\to \cO_T(C)|_C\simeq N_{C/T}\to 0.
\end{equation}
By Lemma \ref{lemvan} and the short exact sequences (\ref{normalbundles}) and (\ref{eqdeCdansT}), the last arrow in (\ref{bigcomposition})  
is surjective, and the third has cokernel of dimension at most $h^2(S,\cO_S)$.
To conclude, it remains to show that the composition of the first two arrows in (\ref{bigcomposition})  
 is surjective. 
 Since the resolution of singularities $T\to\bT$ is an isomorphism along $C$, it coincides with the analogous composition:
\begin{equation}
\label{rewritingcomposition}
T_{B,s}\to H^0(\bT, N_{\bT/\bkX})\to H^0(C, N_{\bT/\bkX}|_C).
\end{equation}
In (\ref{rewritingcomposition}), one has $T_{B,s}=H^0(S, \cL^{\otimes 2}(R))$, $N_{\bT/\bkX}=\cO_{\bkX}(\bT)|_{\bT}=\cO_{\bkX}(2)(R)|_{\bT}$, and the morphism $T_{B,s}\to H^0(\bT, N_{\bT/\bkX})$ describes the variation with $s$ of the equation of $\bT$. Since $\bT=\{r v^2=sw^2\}\subset \bkX$, it is given by multiplication by $w^2$. It follows that (\ref{rewritingcomposition}) is the composition $$H^0(S, \cL^{\otimes 2}(R))\to H^0(C, \cL^{\otimes 2}(R)|_C)\xrightarrow{w^2} H^0(C, \cO_{\bkX}(2)(R)|_C)$$
of the restriction map and of the multiplication by $w^2$.
The former is surjective~as $H^1(S, \cL\otimes\cN)=0$ by Assumption \ref{ass2} and the latter is an isomorphism by (\ref{isov}).
\end{proof}

\subsection{Choice of a birational model of $S$}
\label{biratchoice}
In this section, we explain how to ensure that Assumption \ref{ass1} is satisfied, after replacing $S$ by a blow-up $S'$ at finitely many points lying outside of $R$.

Let $P,Q\subset S$ be two disjoint reduced  finite subschemes of $S$ not meeting $R$, and let $\cI_P$ and $\cI_Q$ be their ideal sheaves. We consider the blow-up $\mu:S'\to S$ of $P\cup Q$, with exceptional divisor $E\cup F\subset S'$ where $E=\mu^{-1}(P)$ and $F=\mu^{-1}(Q)$.

\begin{prop}
\label{twistedlb}
Let $\cA$ be an ample line bundle on $S$, such that $H^1(R,\cA|_R)=0$. Let $\cN:=\mu^*\cA(2F-2E)$.
If $P$ and $Q$ consist of sufficiently many general points, then there exists $\eta\in H^1(S',\cN)$ such that:
\begin{enumerate}[(i)]
\item $H^0(S',\cN\otimes K_{S'})=0$,
\item $H^0(S',\cN^{-1}\otimes K_{S'}(R))=0$,
\item $H^1(R,\cN|_R)=0$,
\item The cup-product map $H^0(S',K_{S'})\xrightarrow{\cupp\eta} H^1(S',\cN\otimes K_{S'})$ is injective
\end{enumerate}
\end{prop}

\begin{proof}
 Condition (iii) follows from our choice of $\cA$.
Since $H^0(S',\cN\otimes K_{S'})=H^0(S',\mu^*(\cA\otimes K_S)(3F-E))=H^0(S,\cA\otimes K_S\otimes \cI_P)$, this group vanishes if $P$ contains sufficiently many general points. By the same argument, $H^0(S',\cN^{-1}\otimes K_{S'}(R))$ vanishes if $Q$ contains sufficiently many general points. 

It remains to check (iv). If $\zeta\in H^0(S,K_S)$, consider the diagram:
\begin{equation}
\begin{aligned}
\label{cupwithdiff}
\xymatrix
@R=0.35cm 
{
H^1(S',\cN)\ar[rr]^(.45){\cupp\mu^*\zeta}&
&  H^1(S',\cN\otimes K_{S'})  \\
H^1(S,\cA\otimes \cI_P^2)\ar@{->>}[r]\ar@{^{(}->}[u]
& H^1(S,\cA\otimes \cI_P)\ar[r]^(.43){\cupp\zeta}&H^1(S,\cA\otimes \cI_P\otimes K_S).\ar@{^{(}->}[u]
}
\end{aligned}
\end{equation}
The vertical maps are edge maps in the Leray spectral sequence for $\mu$, hence are injective. The bottom left horizontal arrow, induced by the inclusion $\cA\otimes\cI_P^2\subset\cA\otimes\cI_P$, is surjective as the cokernel of this inclusion is supported on $P$.  The two other arrows are cup-products by $\zeta$ and by its pull-back $\mu^*\zeta\in H^0(S', K_{S'})$. The description of $\mu^*\zeta$ as induced by the inclusion $\mu^*K_S\subset \mu^*K_S(E+F)\simeq K_{S'}$ shows that (\ref{cupwithdiff}) commutes. Since $\mu^*:H^0(S,K_S)\to H^0(S',K_{S'})$ is an isomorphism,  it suffices to construct $\nu\in H^1(S,\cA\otimes \cI_P)$ such that the 
cup-product map $H^0(S,K_S)\xrightarrow{\cupp\nu} H^1(S,\cA\otimes \cI_P\otimes K_S)$ is injective. Indeed, it follows from (\ref{cupwithdiff}) that a class $\eta\in H^1(S',\cN)$ constructed by lifting $\nu$ to $H^1(S,\cA\otimes \cI_P^2)$ then sending it to $H^1(S',\cN)$ will satisfy the required property.

To construct $\nu$, we consider, for every $\zeta\in H^0(S, K_S)$, the commutative diagram:
\begin{equation}
\begin{aligned}
\label{cupwithdiff2}
\xymatrix
@R=0.3cm 
{
H^0(S,\cA)\ar[r]\ar[d]^{\cupp\zeta}&H^0(P,\cA|_P)\ar[r]^{\partial}\ar[d]^{\cupp\zeta}
&  H^1(S,\cA\otimes \cI_P) \ar[d]^{\cupp\zeta} \\
H^0(S,\cA\otimes K_S)\ar[r]
& H^0(P,(\cA\otimes K_S)|_P)\ar[r]&H^1(S,\cA\otimes \cI_P\otimes K_S),}
\end{aligned}
\end{equation}
and we choose $\nu=\partial(\xi)$ for some $\xi\in H^0(P,\cA|_P)$. We only have to ensure that $\xi\cupp\zeta\in H^0(P,(\cA\otimes K_S)|_P)$ does not belong to the image of $H^0(S,\cA\otimes K_S)$, for every non-zero $\zeta\in H^0(S,K_S)$. This is possible if $P$ contains sufficiently many general points. Indeed, $P$ may then be written as a disjoint union $P=P_1\cup P_2$ such that no non-zero $\zeta\in H^0(S,K_S)$ vanishes identically on $P_1$ and no non-zero section in $ H^0(S,\cA\otimes K_S)$ vanishes identically on $P_2$. Then any $\xi\in H^0(P,\cA|_P)$ that vanishes at every point of $P_2$ but at no point of $P_1$ does the job. 
 \end{proof}

\section{Sufficient conditions for the equality of period and index}
\label{secproof}

We now combine the results of the previous sections to show the equality of the period and the index of some Brauer classes $\alpha\in\Br(\R(S))$ in \S\ref{perind2}--\ref{perind}, and to compute the $u$-invariant of $\R(S)$ in \S\ref{paru}. The only step missing is a control on the topology of the real locus of ramified double covers $p:T\to S$ constructed as in \S\ref{consdouble}. In  \S\ref{topoRdouble}, we gather technical results that will be used for this purpose.

Let $M$ be a one-dimensional $\R$-vector space. The positive elements of $M^{\otimes 2}$ are those of the form $m\otimes m$ for some non-zero $m\in M$. This notion depends on the chosen representation of $M^{\otimes 2}$ as a tensor square, but it will always be clear which one we consider.
 In particular, if $X$ is variety over $\R$ and $\cM$ is a line bundle on $X$, it makes sense to say that a section in $H^0(X,\cM^{\otimes 2})$ is positive at $x\in X(\R)$.

\subsection{Controlling the real locus of a double cover}
\label{topoRdouble}

In \S\ref{topoRdouble}, we fix a connected smooth projective surface $S$ over $\R$, a simple normal crossings divisor $R\subset S$ with equation $r\in H^0(S,\cO_S(R))$, and a union of connected components  $\Xi\subset S(\R)$ such that $R(\R)\subset \Xi$. We will consider the following:

\begin{assumption}
\label{assS}
There exists a compact one-dimensional manifold $\bS$ and a $\ci$ embedding $\iota:\bS\to \Xi$ meeting $R(\R)$ transversally in its smooth locus such that:
\begin{enumerate}[(i)]
\item Every connected component of $\Xi$ meets $\iota(\bS)$.
\item One has $\iota_*[\bS]=\cl_{\R}(R)\in H^1(\Xi)\subset H^1(S(\R))$.
\item The group $H^1(\Xi)$ is generated by classes of connected components of $\iota(\bS)$ and by Borel-Haefliger classes of curves on $S$ whose real locus is included in $\Xi$.
\item Let $\Sigma$ be a connected component of $\Xi$, define $\Sigma_0:=\Sigma\cap(\iota(\bS)\cup R(\R))$  and write $\Sigma\setminus\Sigma_0$ as a disjoint union $\Sigma_+\sqcup\Sigma_-$ of open closed subsets. If any $x\in \Sigma_0$ belongs to the closures of both $\Sigma_+$ and $\Sigma_-$, then $\Sigma_+\cup\Sigma_0$ is connected.
\end{enumerate}
\end{assumption}

The three following lemmas will be useful to verify that Assumption \ref{assS} holds.

\begin{lem}
\label{lemstab}
If $\iota:\bS\to\Xi$ satisfies Assumption \ref{assS}, then so does any $\iota':\bS\to \Xi$ close enough to $\iota$ in $\ci(\bS,\Xi)$ for the strong $\ci$ topology.
\end{lem}

\begin{proof}
That $\iota'$ still satifies conditions (i), (ii) and (iii) is immediate. To show that $\iota'$ still meets $R(\R)$ transversally in its smooth locus and satisfies condition (iv), introduce the normalization $\tR$ of $R$, consider the $\ci$ normal crossings immersion $(\iota,\Id):\bS\sqcup\tR(\R)\to\Xi$, and use \cite[Chapter III, Theorem 3.11, Definition 1.1]{GoGu}.
\end{proof}

\begin{lem}
\label{lemexistiota}
There exists a blow-up $\mu:S'\to S$ at a finite number of general points such that, letting $\Xi':=\mu^{-1}(\Xi)\subset S'(\R)$, Assumption \ref{assS} is satisfied for $(S',\Xi')$.
\end{lem}

\begin{proof}
Consider a union of loops in $\Xi$ whose classes generate $H^1(\Xi)$. Adding additional loops if necessary, and applying $\ci$ approximation and a transversality theorem, 
one obtains a compact one-dimensional manifold $\bS$ and a $\ci$ immersion $\iota :\bS\to \Xi$ meeting every connected component of $\Xi$, intersecting $R(\R)$ transversally at smooth points, that is injective except at finitely many general points of $\Xi$ where it has transverse self-intersection, such that the connected components of $\bS$ generate $H^1(\Xi)$, such that $\iota_*[\bS]=\cl_{\R}(R)\in H^1(\Xi)$, and such that, for every connected component $\Sigma\subset\Xi$, the set $\Sigma\cap(\iota(\bS)\cup R(\R))$ is connected.

Let $\mu:S'\to S$ be the blow-up of $S$ at the finitely many points of transverse self-intersection of $\iota$, so that $\iota$ lifts to an embedding $\iota':\bS\to \Xi'$. We claim that $\iota'$ satisfies the properties required in Assumption \ref{assS}.
Condition (i) is clear.  Computing the cohomology of a blow-up allows to deduce the equality $\iota'_*[\bS]=\cl_{\R}(R)\in H^1(\Xi')$ of condition (ii) from the equality $\iota_*[\bS]=\cl_{\R}(R)\in H^1(\Xi)$. It also shows that $H^1(\Xi')$ is generated by classes of connected components of $\iota'(\bS)$ and by Borel-Haefliger classes of exceptional divisors of $\mu$, yielding (iii).

Let us verify condition (iv). Let $\Sigma'$ be a connected component of $\Xi'$ and let $\Sigma'_0$, $\Sigma'_+$ and $\Sigma'_-$ be as in Assumption \ref{assS} (iv).  Let $x\in \mu(\Sigma)$ be a point of self-intersection of $\iota$. Its preimage in $\Xi'$ (that is the real locus of an exceptional divisor of $\mu$) is a loop meeting transversally $\iota'(\bS)$ at exactly two points. The complement of these two points in the loop has two connected components; one has to belong to $\Sigma'_+$ and the other to $\Sigma'_-$. This shows that these two points belong to the same connected component of  $\Sigma'_+\cup \Sigma'_0$. We then deduce from the connectedness of $\mu(\Sigma'_0)=\mu(\Sigma')\cap(\iota(\bS)\cup R(\R))$ that $\Sigma'_0$ is contained in a unique connected component of $\Sigma'_+\cup \Sigma'_0$. 
The connectedness of $\Sigma'$ then implies that of $\Sigma'_+\cup  \Sigma'_0$, as required.
\end{proof}

\begin{lem}
\label{lemapprox}
 Let $\cA$ and $\cN$ be line bundles on $S$ with $\cA$ very ample, let $\iota:\bS\to \Xi$ be as in Assumption \ref{assS} and let $\cU$ be a neighbourhood of $\iota$ in $\ci(\bS,\Xi)$. 

Then, if $l\in\bN$ is a big enough even integer, there exist elements $\iota'\in\cU$ and $t\in H^0(S,\cA^{\otimes l}\otimes\cN^{\otimes 2}(-R))$ such that $\{t=0\}$ is smooth along its real locus $Z$, such that $Z=\iota'(\bS)$, and such that $rt\in  H^0(S,\cA^{\otimes l}\otimes\cN^{\otimes 2})$ is negative on $S(\R)\setminus\Xi$.
\end{lem}

\begin{proof}
By \cite[Theorem 12.4.11]{BCR}, there exist $\iota'\in \cU$ and a hypersurface $H\subset S$ that is smooth along $S(\R)$ and such that $H(\R)=\iota'(\bS)$. Since $\iota_*[\bS]=\cl_{\R}(R)$, one has $\cl_{\R}(R+H)=0$. Consequently, one may apply Br\"ocker's theorem \cite[Satz b)]{Brocker}
to find $h_1\in\R(S)^*$ that vanishes at order one along $R$ and $H$ and that has no other zeros or poles along $S(\R)$. 
By the Stone-Weierstrass theorem (see \cite[Satz a)]{Brocker}), there exists $h_2\in\R(S)^*$ invertible along $S(\R)$ that has the same signs as $h_1$ on $S(\R)\setminus \Xi$.
Let $D$ be the divisor of poles of $h_1h_2$, let $l$ be an even number such that $\cA^{\otimes l/2 }\otimes\cN(-R-D)$ is very ample. Let $(u_i)$ be a basis of $H^0(S,\cA^{\otimes l/2}\otimes\cN(-R-D))$ viewed as sections of $\cA^{\otimes l/2}\otimes\cN(-R)$ vanishing on $D$, so that $u:=\sum_i u_i^2\in H^0(S,\cA^{\otimes l}\otimes\cN^{\otimes 2}(-2R))$. Then $t:=-h_1h_2ru$ works.
\end{proof}

We now explain how Assumption \ref{assS} will be used to control the topology of the real locus of a ramified double cover.

\begin{lem}
\label{fromStoT}
Let $\cL$ be a line bundle on $S$, let $s\in H^0(S,\cL^{\otimes 2}(R))$ be a section with smooth zero-locus $D$, such that $rs\in H^0(S,\cL^{\otimes 2}(2R))$ is negative on $S(\R)\setminus\Xi$. Let $p:T\to S$ be the double cover ramified over $\Delta:=R\cup D$ constructed in \S\ref{consdouble}. Define $\bS:=D(\R)$ and assume that the inclusion $\iota:\bS\to \Xi$ satisfies Assumption \ref{assS}.
Then $p:T\to S$ satisfies Assumption \ref{assumkilltop}.
\end{lem}

\begin{proof} Since $rs$ is negative on   $S(\R)\setminus \Xi$, the equation of $\bT$ given in \S\ref{consdouble} shows that $p(T(\R))\subset \Xi$. In particular, $p(T(\R))$ is disjoint from $\Theta$, and Assumption \ref{assumkilltop} (i) holds. Since $\iota(\bS)=D(\R)$ and $D$ lifts to $T$, the map $\iota$ lifts to a $\ci$ map $\bS\to T(\R)$, and Assumption \ref{assumkilltop} (ii) follows from Assumption \ref{assS} (iii).

 Let $\Sigma\subset \Xi$ be a connected component, and let $p^{-1}(\Sigma)\subset T(\R)$ be its preimage in $T(\R)$. To verify Assumption \ref{assumkilltop} (iii), we will show that $p^{-1}(\Sigma)$ is connected.
 The equation of $\bT$ given in \S\ref{consdouble} shows that $p(p^{-1}(\Sigma))=\{x\in\Sigma\mid rs(x)\geq 0\}$.
 Define $\Sigma_+:=\{x\in\Sigma\mid rs(x)>0\}$ and $\Sigma_-:=\{x\in\Sigma\mid rs(x)<0\}$. By Assumption \ref{assS} (iv), $p(p^{-1}(\Sigma))=\Sigma_+\cup(\Sigma\cap\Delta(\R))$ is connected. Suppose for contradiction that $p^{-1}(\Sigma)$ is not connected. Since $p(p^{-1}(\Sigma))$ is connected, one may find two distinct connected components $\Pi_1,\Pi_2\subset p^{-1}(\Sigma)$ such that $\Pi:=p(\Pi_1)\cap p(\Pi_2)\neq\varnothing$. By Assumption \ref{assS} (i), $\Sigma_-\neq\varnothing$, so that $\Pi\neq\Sigma$. Since $\Pi$ is closed, it is not open, and one may choose $x\in \Pi$ not belonging to the  interior of $\Pi$. Since $\Pi_1\to \Sigma$ and $\Pi_2\to \Sigma$ cannot be both open above $x$, one has $x\in\Delta(\R)$. Since the fibers of $T(\R)\to S(\R)$ above a point in $\Delta(\R)$ are connected, $\Pi_1$ and $\Pi_2$ intersect, a contradiction. 
 Thus, $p^{-1}(\Sigma)$ is connected, and Assumption \ref{assumkilltop} (iii) follows.
\end{proof}

\subsection{Brauer classes of period $2$}
\label{perind2}

We prove, for Brauer classes of period $2$, a statement that is slightly more general than our main theorems, and that will be useful to handle Brauer classes of higher period.

\begin{prop}
\label{propperiod2}
Let $S$ be a connected smooth projective surface over $\R$, and let $\alpha\in\Br(\R(S))[2]$ be a Brauer class of period $2$.
Let $U\subset S$ be the biggest open subset such that $\alpha\in\Br(U)\subset\Br(\R(S))$ as in \S\ref{parBrauer}, and let $\talpha\in H^2_G(U(\C),\Z/2)$ be a lift of $\alpha$ in (\ref{Brlift}). Define $\Theta:=\{x\in U(\R)\mid \alpha|_x\neq 0\in\Br(\R)\}$. 
Write $S(\R)=\Xi\sqcup\Psi$ as a disjoint union of open closed subsets
such that $\Theta\subset \Psi\subset U(\R)$
and $([\talpha]_1)|_{\Psi}=0\in H^1(\Psi)$.
Then there exists a connected smooth projective surface $T$ over $\R$ and a degree $2$ morphism $p:T\to S$ such that $p(T(\R))\subset\Xi$ and $\alpha_{\R(T)}=0\in\Br(\R(T))$.
\end{prop}

\begin{proof}
Replacing $S$ with a modification, we may assume that the ramification locus $R:= S\setminus U$ of $\alpha$ is a simple normal crossings divisor. Let $r\in H^0(S,\cO_S(R))$ be an equation of $R$. By Proposition \ref{twistedlb}, we may assume that $S$ carries a line bundle $\cN$ satisfying Assumption \ref{ass1}, after blowing up $S$ at finitely many points outside of $R$. Proposition \ref{twistedlb} ensures moreover that such a line bundle will still exist if we blow-up $S$ further at finitely many general points outside of $R$. By  Lemma \ref{lemexistiota}, after such a blow-up, we may suppose that there exists an embedding $\iota:\bS\to \Xi$ as in Assumption \ref{assS}.

Let $\cA$ be a very ample line bundle on $S$. Define $\cL:=\cA^{\otimes l}\otimes\cN(-R)$ for a sufficiently big even integer $l\in\bN$, so that Assumptions \ref{assumample} and \ref{ass2} hold, and so that Lemma \ref{lemapprox} may be applied. Lemma \ref{lemapprox} then shows, up to replacing $\iota$ by a small deformation which is legitimate by Lemma \ref{lemstab},
 the existence of $t\in H^0(S,\cL\otimes\cN)$ such that $\{t=0\}$ is smooth along its real locus, equal to $\iota(\bS)$, and such that $rt\in H^0(S, \cA^{\otimes l}\otimes\cN^{\otimes 2})$ is negative
on $\Psi$. 

We apply the construction of \S\ref{consdouble}, which produces a family $\pi:\cT\to B$ of surfaces that are both ramified double covers of $S$ and hypersurfaces in the smooth projective threefold $\kX$ over $\R$. In \S\ref{NLlocus}, we constructed sections $c\in H^0(S,\cA^{\otimes l})$ and $s_0\in H^0(S,\cL^{\otimes 2}(R))$, and showed that a general $s\in H^0(S,\cL^{\otimes 2}(R))$ of the form $a_1s_0+a_2c$ 
with $a_1\in\R_{>0}$ and $a_2\in H^0(S,\cL\otimes\cN)$ corresponds to a point $s\in B(\R)$ whose associated surface $T:=\cT_s$ contains a particular curve $C\subset T$. We choose such a general section $s$ with $a_1$ sufficiently small and $a_2$  sufficiently close to $t$. As $c$ is positive on $S(\R)$, this ensures that $rs$ is negative on  $\Psi$.
Thanks to Lemma \ref{lemstab},
this also ensures, after modifying $\iota$ again, that $\iota(\bS)$ is the real locus of the smooth divisor $D:=\{s=0\}$. By Lemma \ref{fromStoT}, $p:T\to S$ satisfies Assumption \ref{assumkilltop}.

Having ensured that Assumptions \ref{assumample} and \ref{assumkilltop} hold, 
and in view of our hypothesis that
$([\talpha]_1)|_{\Psi}=0$, we may apply successively Propositions \ref{ramiprop},
\ref{lem1R} and \ref{lem2R}.
This shows that $\alpha_{\R(T)}\in\Br(T)[2]\subset\Br(\R(T))[2]$, that this Brauer class is induced by a class $\beta\in H^2_G(T(\C),\Z(1))$ in (\ref{Brdeco}), and that there exist $\gamma\in H^2_G(T(\C),\Z(1))$ and $\theta\in\Pic(S)$ such that $p_*(\beta-2\gamma)=\cl(\theta)\in H^2_G(S(\C),\Z(1))$.

In the remainder of the proof, we apply \cite[\S 1]{NL3} to the family $\pi:\cT\to B$ of smooth projective surfaces over $\R$. We still denote by $p:\cT\to S$ the natural morphism that realizes the fibers of $\pi$ as ramified double covers of $S$.
The $G$-equivariant local system $\bH^2_\Q:=\RR^2\pi_*\Q$ splits as a direct sum $\bH^2_\Q=H^2(S(\C),\Q)\oplus\bH^2_{\Q,\van}$, where $H^2(S(\C),\Q)$ is the constant sub-local system induced by $p$, and $\bH^2_{\Q,\van}$ is its orthogonal with respect to the cup-product. It carries a variation of Hodge structures: the holomorphic bundle $\cH^2:=\bH^2_\Q\otimes_{\Q}\cO_{B(\C)}$ is endowed with a Gauss-Manin connection $\nabla:\cH^2\to\cH^2\otimes\Omega^1_{B(\C)}$ and with a Hodge filtration.
 The sub-local systems $H^2(S(\C),\Q)$ and $\bH^2_{\Q,\van}$ are sub-variations of Hodge structures, and $H^2(S(\C),\Q)$ is a constant one.

Let $\lambda\in H^1(T,\Omega^1_T)$ be the class of $C$ in Hodge cohomology. By \cite[Remark 1.4]{NL3}, it belongs to $H^{1,1}_{\R}(T(\C))(1)^G:=[H^2(T(\C),\R(1))\cap H^{1,1}(T(\C))]^G$, where $G$ acts both on $T(\C)$ and on $\R(1)$. Griffiths \cite[Th\'eor\` eme 10.21]{voisinbook} has computed that the map $\overline{\nabla}(\lambda):T_{B(\C),s}\to H^2(T,\cO_T)$
 induced by evaluating $\nabla$ at $\lambda$ is exactly (the complexification of) the map $\phi_{\lambda}$ defined in (\ref{defpsilambda}). By Proposition \ref{propverifGreen}, its cokernel has dimension at most $h^2(S,\cO_S)$. Since the sub-Hodge structure $H^2(S(\C),\Q)\subset \bH^2_{\Q}$ is constant, the image of $\overline{\nabla}(\lambda)$ is included in the second factor of the decomposition $H^2(T,\cO_T)=H^2(S,\cO_S)\oplus H^2(T,\cO_T)_{\van}$. It follows that $\overline{\nabla}(\lambda):T_{B(\C),s}\to H^2(T,\cO_T)_{\van}$ is surjective.

Choose a $G$-stable connected  analytic neighbourhood $\Lambda$ of $s$ in $B(\C)$ on which $\bH^2_{\Q,\van}$ is trivialized and such that $\Lambda(\R):=\Lambda\cap B(\R)$ is connected and contractible, as in \cite[\S 1.2]{NL3}. Define $\cT(\C)|_{\Lambda(\R)}$ and $\cT(\R)|_{\Lambda(\R)}$ as the inverse images of $\Lambda(\R)$ by $\pi:\cT(\C)\to B(\C)$ and $\pi:\cT(\R)\to B(\R)$. 
By Assumption \ref{assumkilltop} (i), we may assume after shrinking $\Lambda(\R)$ that $p(\cT(\R)|_{\Lambda(\R)})\subset \Xi$.
By a $G$-equivariant version of Ehresmann's theorem \cite[Lemma 4]{Dimca}, we may assume after further shrinking $\Lambda(\R)$ that there exists a $G$-equivariant isomorphism $\cT(\C)|_{\Lambda(\R)}\xrightarrow{\sim} T(\C)\times \Lambda(\R)$ respecting the projection to $\Lambda(\R)$. From now on, using this $G$-equivariant isomorphism, we identify the Betti cohomology groups and the equivariant Betti cohomology groups of the fibers of $\pi:\cT(\C)|_{\Lambda(\R)}\to \Lambda(\R)$.

By the real analogue of Green's infinitesimal criterion \cite[Proposition 1.1]{NL3} applied to $\bH^2_{\Q,\van}$,
 there exists an open cone $\Omega\subset H^2(T(\C),\R(1))_{\van}^G$ with the property that for every $\nu\in\Omega$, there exists $x\in\Lambda(\R)$ such that $\nu$ is of type $(1,1)$ in the Hodge decomposition of $H^2(\cT_x(\C),\C)$.

Let $(\vep_1,\vep_2)$ be the image of $\beta-2\gamma\in H^2_G(T(\C),\Z(1))$ in the decomposition
\begin{equation}
\label{decovan}
H^2(T(\C),\R(1))^G=H^2(S(\C),\R(1))^G\oplus H^2(T(\C),\R(1))_{\van}^G.
\end{equation}
Let $H^2_G(T(\C),\Z(1))_{\van}$ be the subgroup of $H^2_G(T(\C),\Z(1))$ consisting of classes whose images in the first factor of (\ref{decovan})
vanish. Since the image of $H^2_G(T(\C),\Z(1))$ in $H^2(T(\C),\Z(1))^G$ has finite index by the Hochschild-Serre spectral sequence, the image of  $H^2_G(T(\C),\Z(1))_{\van}$ in $H^2(T(\C),\R(1))_{\van}^G$ is a lattice. 
Since moreover $\Omega\subset H^2(T(\C),\R(1))_{\van}^G$ is an open cone, one may find $\delta\in H^2_G(T(\C),\Z(1))_{\van}$ and $\nu\in \Omega$ such that $\vep_2=2\delta+\nu$.  By definition of $\Omega$, there exists $x\in \Lambda(\R)$ such that $\nu$ is of type $(1,1)$ in $H^2(\cT_x(\C),\C)$.
The equality $p_*(\beta-2\gamma)=\cl(\theta)$ shows that $\vep_1$ is of type $(1,1)$ in $H^2(S(\C),\C)$.
The above facts show that $\beta-2\gamma-2\delta\in H^2_G(\cT_x(\C),\Z(1))$ is a class of type $(1,1)$. By the real Lefschetz (1,1) theorem \cite[Proposition 2.8]{BWI}, there exists $\varphi\in\Pic(\cT_x)$ such that  $\beta=2\gamma+2\delta+\cl(\varphi)\in H^2_G(\cT_x(\C),\Z(1))$. By the exact sequence (\ref{Brdeco}), replacing $T$ with $\cT_x$ proves the proposition.
\end{proof}

\subsection{Brauer classes of arbitrary period}
\label{perind}

It is now possible to complete the proofs of Theorems \ref{pisanspoint} and \ref{piavecpoint}, and of the first half of Theorem \ref{pinr}. 

\begin{prop}
\label{proofarbitraryperiod}
Let $S$ be a connected smooth projective surface over $\R$, and let $\alpha\in\Br(\R(S))[n]$ be of period $n$.
Let $U\subset S$ be, as in \S\ref{parBrauer}, the biggest open subset such that $\alpha\in\Br(U)\subset\Br(\R(S))$, and let $\xi\in H^2_G(U(\C),\Z/2)$ be a lift of $\frac{n}{2}\alpha$ in (\ref{Brlift}). Assume that $\Theta:=\{x\in U(\R)\mid \alpha|_x\neq 0\in\Br(\R)\}$ is a union of connected components of $S(\R)$, and that $([\xi]_1)|_{\Theta}=0\in H^1(\Theta)$. Then $\ind(\alpha)=n$.
\end{prop}

\begin{proof}
If $n$ is odd, that $\ind(\alpha)=n$ follows from de Jong's theorem \cite{deJong} and a norm argument.  We now suppose that $n$ is even and argue by induction on $n$.

We apply Proposition \ref{propperiod2} to the period $2$ class $\frac{n}{2}\alpha\in\Br(\R(S))[2]$, with $\Psi=\Theta$ and $\Xi=S(\R)\setminus\Theta$. We deduce the existence of a degree $2$ morphism $p:T\to S$ between connected smooth projective surfaces over $\R$ such that $\alpha_{\R(T)}\in\Br(\R(T))$ has period $\frac{n}{2}$, and such that $p(T(\R))$ is disjoint from $\Theta$. Let $V\subset T$ be the biggest open subset  such that $\alpha_{\R(T)}\in\Br(V)\subset\Br(\R(T))$, as in \S\ref{parBrauer}. Since $p(T(\R))$ and $\Theta$ do not intersect, $\alpha_{\R(T)}|_x=0$ for every $x\in p^{-1}(U)(\R)$. As this property only depends on the connected component of $V(\R)$ to which $x$ belongs, the same holds for every $x\in V(\R)$. By the induction hypothesis,  $\alpha_{\R(T)}$ has index $\frac{n}{2}$. It follows, as wanted, that $\ind(\alpha)=n$.
\end{proof}

Applying Proposition \ref{proofarbitraryperiod} when $\Theta=\varnothing$ gives a proof of Theorem \ref{piavecpoint}.
Theorem \ref{pisanspoint} is the even more particular of Proposition \ref{proofarbitraryperiod} when $S(\R)=\varnothing$. The first half of Theorem \ref{pinr} also follows from Proposition \ref{proofarbitraryperiod}, applied when $U=S$. We will complete the proof of Theorem \ref{pinr} in \S\ref{subobstr}.

\subsection{The $u$-invariant of function fields of real surfaces}
\label{paru}

We now explain why Theorem \ref{uinvariant} follows from Theorem \ref{piavecpoint}.
We pointed out in \S\ref{parLang} that Theorems \ref{quadrique} and \ref{dP4} are, in turn, consequences of Theorem \ref{uinvariant}.

\begin{proof}[Proof of Theorem \ref{uinvariant}]
Let $\R\subset K$ be an extension of transcendence degree $2$. By Pfister's criterion  \cite[Proposition 9]{Pfisterabelian}, showing that $u(K)\leq 4$ is equivalent to proving that every non-zero $\alpha\in \Br(K)[2]$ such that $\alpha|_{\bK}=0$ for all real closures $K\subset \bK$ has index $2$.
Let $\R\subset K_0\subset K$ be a subfield of $K$ finitely generated over $\R$ such that $\alpha$ is the image of a class $\alpha_0\in\Br(K_0)[2]$.
Write $K=\cup_i K_i$ as the union of all finite extensions of $K_0$, and let $\alpha_i\in\Br(K_i)[2]$ be the image of $\alpha_0$. Let $X_i$ be the space of orderings of $K_i$ endowed with the Harrison topology \cite[VIII, \S 6]{Lam}.
The subset $Z_i\subset X_i$ of orderings such that $\alpha_i$ does not vanish on the associated real closure is closed by \cite[Hilfssatz 2]{Arason},
 hence compact \cite[VIII, Theorem 6.3]{Lam}. Since $\varprojlim_i Z_i=\varnothing$ by our hypothesis on $\alpha$,  there exists $i$ such that $Z_i=\varnothing$ by Tychonoff's theorem. Replacing $K$ with $K_i$ and $\alpha$ with $\alpha_i$, we may assume from now on that $K$ is the function field of an integral surface $S$  over $\R$.

Let $U\subset S$ be an open subset such that $\alpha\in \Br(U)\subset\Br(\R(S))$. Let us prove that $\alpha|_x=0\in\Br(\R)$ for every $x\in U(\R)$. To do so, we choose a local system of parameters $z_1,z_2\in\cO_{S,x}$ of $S$ at $x$, and let $\R((z_1,z_2))\subset\bK$ be any real closed extension, such as $\bK=\cup_{n_2}\big[\cup_{n_1}\R((z_1^{1/n_1}))\big]((z_2^{1/n_2}))$. In the
diagram $\Br(\R)\to\Br(\R[[z_1,z_2]])\to\Br(\bK)$, the first arrow is an isomorphism
by proper base change and has a retraction given by restriction to $x$, and the composition of the two arrows is an isomorphism as both $\Br(\R)$ and $\Br(\bK)$ are generated by the  quaternion class $(-1,-1)$. That $\alpha_{\bK}=0$ implies at once that $\alpha|_x=0$, as wanted.
Theorem \ref{piavecpoint} now shows that $\alpha$ has index $2$, as wanted.

To prove the easier inequality $u(\R(S))\geq 4$, let $x\in S$ be a closed point  with residue field $\kappa(x)$ isomorphic to $\C$, and let $\mathfrak{m}_{S,x}\subset\cO_{S,x}$ be the maximal ideal. Choose $w\in\cO_{S,x}$ inducing $\sqrt{-1}\in\C\simeq\kappa(x)$ and let $z_1,z_2\in \cO_{S,x}$ be a local system of parameters at $x$ such that neither $z_1$ nor $z_2$ is proportional to $w^2+1$ in $\mathfrak{m}_{S,x}/\mathfrak{m}_{S,x}^2$. Define $y_j:=(1+z_j)^2+(w+z_j)^2$, so that  $y_1,y_2\in \R(S)$ also forms a local system of parameters at $x$. Completion at $x$ induces an inclusion $\R(S)\subset\C((y_1))((y_2))$. By Springer's results on quadratic forms over complete discrete valuation fields \cite{Springer} (see \cite[VI, Proposition 1.9 (2)]{Lam}), the quadratic form $\langle 1,y_1,-y_2,-y_1y_2\rangle$ is anisotropic over $\C((y_1))((y_2))$, hence over $\R(S)$. Since $y_1$ and $y_2$ are positive with respect to every ordering of $\R(S)$, one deduces that $u(\R(S))\geq 4$.
\end{proof}

\section{An obstruction to the equality of period and index}
\label{secpush}

The main result of this section is the following proposition, proven in \S\ref{subobstr}.

\begin{prop}
\label{topoobstr}
 Let $p:T\to S$ be a morphism of connected smooth projective surfaces over $\R$ that is generically finite of even degree $n$. Let $\alpha\in\Br(\R(S))[n]$ be such that $p^*\alpha=0\in \Br(\R(T))$. Let $U\subset S$ be the biggest open subset such that $\alpha\in\Br(U)\subset\Br(\R(S))$, as in \S\ref{parBrauer}, define $\Theta:=\{x\in U(\R)\ |\ \alpha|_x\neq 0\in\Br(\R)\}$, let ${\talpha\in H^2_G(U(\C),\Z/n(1))}$ be a lift of $\alpha$ in (\ref{Brlift}), and let $\xi\in H^2_G(U(\C),\Z/2)$ be its reduction modulo $2$.
Then there exists $\nu\in\Pic(S)$ with ${([\xi]_1)|_\Theta=\cl_{\R}(\nu)|_\Theta \in H^1(\Theta)}$.
\end{prop}

The conclusion of Proposition \ref{topoobstr} should be thought of as an obstruction to the index of $\alpha$ dividing $n$. It will be used in \S\ref{subobstr} to finish the proof of Theorem \ref{pinr}.

We keep the notation of the statement of Proposition \ref{topoobstr} throughout Section \ref{secpush}.

\subsection{The topology of a ramified cover}
\label{toposubsec}
We first note for later use that the vanishing $p^*\alpha=0\in\Br(\R(T))$ implies that $p(T(\R))\cap\Theta=\varnothing$.

Let $j:S^{*}\hookrightarrow S$ be the open subset over which $p$ is \'etale.
Let $\sQ$ be the cokernel of the natural injection $p^*:\Z\to p_*\Z$ of $G$-equivariant sheaves on $S(\C)$. Note that the adjunction map $\Z\to j_*j^*\Z$ (resp. $p_*\Z\to j_*j^*p_*\Z$) is an isomorphism because removing $S(\C)\setminus S^*(\C)$ (resp. $p^{-1}(S(\C)\setminus S^*(\C))$) does not disconnect $S(\C)$ (resp. $T(\C)$) locally. 
We deduce that the adjunction map $\sQ\to j_*j^*\sQ$ is injective.

The trace map $\Tr: p_*\Z\to \Z$, defined over $S^*(\C)$ by summing over the fiber, extends uniquely to $S(\C)$ because $\Z\xrightarrow{\sim}j_*j^*\Z$. It induces a morphism $\phi:\sQ\to\Z/n$, yielding a commutative diagram with exact rows of $G$-equivariant sheaves on $S(\C)$:
\begin{equation}
\begin{aligned}
\label{diagQ}
\xymatrix
@R=0.3cm 
{
0 \ar[r]&\Z\ar[r]^{p^*}\ar@{=}[d]&p_*\Z\ar[r]^{\pi}\ar[d]^{\Tr}
&  \sQ \ar[d]^{\phi}\ar[r] &0\\
0\ar[r]&\Z\ar[r]^n
& \Z \ar[r]&\Z/n\ar[r] &0.}
\end{aligned}
\end{equation}
The diagram (\ref{diagQ}) induces a short exact sequence of $G$-equivariant sheaves on $S(\C)$:
\begin{equation}
\label{shortramified}
0\to p_*\Z\xrightarrow{(\Tr,\pi)}\Z\oplus\sQ\xrightarrow{(1,-\phi)}\Z/n\to 0.
\end{equation}

For $x\in S^*(\C)$, if we choose a bijection between $p^{-1}(x)$ and $\{1,\dots,n\}$, the fiber $\sQ_x$ of $\sQ$ at $x$ identifies with the cokernel of $\Z\xrightarrow{a\mapsto(a,\dots, a)}\Z^n$. It follows that the assignment $(a_1,\dots,a_n)\mapsto (\sum_i a_i-na_1,\dots,\sum_i a_i-na_n)$ yields a well-defined morphism $j^*\sQ\to j^*p_*\Z$ of $G$-equivariant sheaves on $S^*(\C)$. This morphism extends to a morphism $\psi:\sQ\to p_*\Z$ of $G$-equivariant sheaves on $S(\C)$ 
because $p_*\Z\xrightarrow{\sim}j_*j^*p_*\Z$. The morphism $\Tr\circ\psi:\sQ\to \Z$ vanishes on $S^*(\C)$ as one checks by computing its stalks, hence on $S(\C)$, because $\Z\xrightarrow{\sim}j_*j^*\Z$.

\subsection{The alternating double cover}
\label{alternating}
Keep the notation of \S\ref{toposubsec}.
Let $x_0\in S^*(\C)$ be a base point, choose a bijection between $p^{-1}(x_0)$ and $\{1,\dots,n\}$, and consider the representation $\pi_1^{\et}(S^*,x_0)\to\mathfrak{S}_n$ associated with the finite \'etale cover $p^{-1}(S^*)\to S^*$. Composing  with the signature morphism $\mathfrak{S}_n\to\Z/2$, one obtains a representation $\pi_1^{\et}(S^*,x_0)\to\Z/2$ that corresponds to a finite \'etale double cover $\hat{p}:\hat{S}^*\to S^*$. By construction, a point $y\in \hat{S}^*(\C)$ is uniquely determined by $x=\hat{p}(y)\in S^*(\C)$, and by a  bijection between $p^{-1}(x)$ and $\{1,\dots,n\}$ well-defined up to the action of the alternating group $\mathfrak{A}_n$. This double cover extends uniquely to a finite double cover with normal total space $\hat{p}:\hat{S}\to S$.

 Let us compute the biggest open subset $V\subset S$ over which $\hat{p}$ is unramified.
Let $\Delta\subset S$ be the ramification divisor of $p$, with irreducible components $(\Delta_i)_{i\in I}$, and let $F_i$ be the geometric fiber of $p$ at the generic point of $\Delta_i$. For $i\in I$, define the multiplicity $m_i$ of $\Delta_i$ by the formula $m_i=n-|F_i|$. It is equivalently computed as $m_i=\sum_{y\in F_i} (e(y)-1)$, where $e(y)$ is the ramification index of $p$ at $y$. Since the monodromy around a component of $\Delta_i(\C)$ is an even permutation if and only if $m_i$ is even, and since the ramification locus of $\hat{p}$ has pure dimension $1$ by the Zariski-Nagata purity theorem, one has $V=S\setminus(\underset{m_i\textrm{ odd }}{\cup}\Delta_i)$. We denote by $\hp:\hV\to V$ the induced finite \'etale double cover.

 Let $U^0\subset U$ be the biggest open subset over which $p$ is finite flat with smooth ramification locus. It is the complement of finitely many points in $U$. Define $\Theta^0:=\Theta\cap U^0(\R)$. 
We claim that $\Theta^0\subset V(\R)$. Otherwise, there would exist $i\in I$ such that $m_i$ is odd and $x\in \Delta_i(\R)\cap\Theta^0$. Since $\Theta^0\subset S(\R)$ is open and since $\Delta_i$ is smooth along $\Theta^0$, we may assume that $x$ is a general point of $\Delta_i$. The geometric fiber of $p$ over $x$ then has cardinality $n-m_i$. Since it is an odd number, we would have $x\in p(T(\R))$ contradicting the fact that $p(T(\R))\cap\Theta=\varnothing$.

\subsection{The pull-back of the Brauer class}
\label{subspullback}
Keep the notation of \S\S\ref{toposubsec}--\ref{alternating}, and
 denote by $p:T_{U^0}\to U^0$ the base-change of $p$ by the inclusion $U^0\subset S$.

\begin{lem}
\label{etazeta}
The class $\talpha|_{U^0}\in H^2_G(U^0(\C),\Z/n(1))$ is the image under the morphism $(1,-\phi)$ of (\ref{shortramified}) of a class $(\eta,\zeta)\in H^2_G(U^0(\C),\Z(1))\oplus H^2_G(U^0(\C),\sQ(1))$.
\end{lem}

\begin{proof}
We have to show that $\talpha|_{U^0}$ vanishes in $H^3_G(U^0(\C),p_*\Z(1))$ under the boundary map associated to (\ref{shortramified}). Since $p$ is finite over $U^0$, the Leray spectral sequence of $p:T^{}_{U^0}\to U^0$ degenerates, and the natural morphism $H^3_G(U^0(\C),p_*\Z(1))\to H^3_G(T_{U^0}(\C),\Z(1))$ is an isomorphism. The commutative exact diagram
\begin{equation*}
\begin{aligned}
\xymatrix
@R=0.4cm 
{
0 \ar[r]&\Z(1)\ar[r]^n\ar[d]^{p^*}&\Z(1)\ar[r]\ar[d]^{(1,0)}
&  \Z/n(1) \ar@{=}[d]\ar[r] &0\\
0\ar[r]&p_*\Z(1)\ar[r]^{(\Tr,\pi)\hspace{1.1em}}
& \Z(1)\oplus\sQ(1) \ar[r]^{(1,-\phi)}&\Z/n(1)\ar[r] &0}
\end{aligned}
\end{equation*}
shows that the image of $\talpha|_{U^0}$ in $H^3_G(T_{U^0}(\C),\Z(1))$ may be computed as the image of $p^*\alpha|_{U^0}\in\Br(T_{U^0})$ by (\ref{Brdeco}), that vanishes by our assumption on $\alpha$.
\end{proof}

 From now on, fix $(\eta,\zeta)\in H^2_G(U^0(\C),\Z(1))\oplus H^2_G(U^0(\C),\sQ(1))$ as in Lemma \ref{etazeta}.
Recall the definition of $\psi:\sQ\to p_*\Z$ in \S\ref{toposubsec}. We still denote by $\psi$ the composition:
$$ H^2_G(U^0(\C),\sQ(1))\xrightarrow{\psi}H^2_G(U^0(\C),p_*\Z(1))
\to H^2_G(T_{U^0}(\C),\Z(1)).$$

\begin{lem}
\label{etazetauseful}
The class $p^*\talpha|_{U^0}\in H^2_G(T_{U^0}^{}(\C),\Z/n(1))$ coincides with the reduction modulo~$n$ of 
$p^*\eta+\psi(\zeta)\in H^2_G(T_{U^0}^{}(\C),\Z(1))$.
\end{lem}

\begin{proof}
The statement follows from the commutativity of the two diagrams:
\begin{equation}
\label{diageta}
\begin{aligned}
\xymatrix
@R=0.3cm 
{
H^2_G(U^0(\C),\Z(1))\ar[r]\ar[d]^{p^*}&H^2_G(U^0(\C),\Z/n(1))\ar[d]^{p^*} &\\
H^2_G(T_{U^0}^{}(\C),\Z(1))\ar[r]
& H^2_G(T_{U^0}^{}(\C),\Z/n(1)) & \text {and}
}
\end{aligned}
\end{equation}
\vspace{-0.5em}
\begin{equation}
\label{diagzeta}
\begin{aligned}
\xymatrix
@R=0.3cm 
{
H^2_G(U^0(\C),\sQ(1))\ar[r]^{\phi\hspace{1em}}\ar[d]^{\psi}&H^2_G(U^0(\C),\Z/n(1))\ar[d]^{p^*}& \\
H^2_G(T_{U^0}(\C),\Z(1))\ar[r]
& H^2_G(T_{U^0}(\C),\Z/n(1)),&
}
\end{aligned}
\end{equation}
applied to $\eta$ and $\zeta$ respectively. The commutativity of (\ref{diageta}) is obvious. To show that of (\ref{diagzeta}), consider the two morphisms of $G$-equivariant sheaves on $S(\C)$ given by $p^*\circ\phi: \sQ\to p_*(\Z/n)$ and $(\psi \mod n): \sQ\to p_*(\Z/n)$. They coincide on $S^*(\C)$ because their stalk at $x\in S^*(\C)$ are both given by the assignment $(a_1,\dots,a_n)\mapsto (\sum_i a_i,\dots,\sum_i a_i)$. Since $\sQ\to j_*j^*\sQ$ is injective, they coincide over $S(\C)$. The commutativity of (\ref{diagzeta}) results by taking equivariant cohomology.
\end{proof}

\subsection{An algebraicity result}
\label{algebraicitypar}
We keep the notation of \S\S\ref{toposubsec}--\ref{subspullback}. Recall that $\zeta\in H^2_G(U^0(\C),\sQ(1))$ has been constructed in \S\ref{subspullback}. In \S\ref{algebraicitypar}, we study its image $\phi(\zeta)\in H^2_G(U^0(\C),\Z/n(1))$ by $\phi$.  
Our main goal is  the following proposition, that is key in proving the second part of Theorem \ref{pinr}.

\begin{prop}
\label{algebraicity}
 One has $([\phi(\zeta)]_1)|_{\Theta^0}\s=\cl_{\R}(\theta)|_{\Theta^0}\in H^1(\Theta^0)$ for some $\theta\in \Pic(S)$.
\end{prop}

Let $\iota:\bS^1\to\Theta^0$ be a $\ci$ embedding meeting the ramification locus of $p$ transversally at general points.
By abuse of notation, we will still denote by $\iota$ the induced embedding of $\bS^1$ in $S(\C)$, $S(\R)$, etc.
Define a sheaf $\sG$ on $\bS^1$ by the exact sequence:
\begin{equation}
\label{defG}
\iota^*\sQ(1)\xrightarrow{(\phi_2 ,1+\sigma)} \Z/2\oplus(\iota^* \sQ(1))^{G}\to\sG\to 0,
\end{equation}
where $\phi_2:\sQ(1)\to\Z/2$ is the composition of $\phi:\sQ(1)\to\Z/n(1)$ and of the surjection $\Z/n(1)\to\Z/2$.  
We consider the diagram of sheaves on $\bS^1$:
\begin{equation}
\label{Gext}
0\to\Z/2\xrightarrow{(1,0)}\sG\xrightarrow{(0,\phi)}\Z/2\to 0,
\end{equation}
where we still denote by $\phi $ the restriction $(\iota^*\sQ(1))^{G}\to(\Z/n(1))^{G}\simeq \Z/2$ of $\phi$.
\begin{lem}
\label{Gzeta}
The diagram (\ref{Gext}) is an exact sequence of sheaves on $\bS^1$. Moreover, the image  of $1\in H^0(\bS^1)$ by its boundary map is $\iota^*[\phi(\zeta)]_1\in H^1(\bS^1)$.
\end{lem}

\begin{proof}
If $x\in\bS^1$, one has $\sExt^q_G(\Z,\iota^*p_*\Z)_x=H^q(G,(\iota^*p_*\Z)_x)=0$ for $q>0$. Indeed, the first equality is explained in \cite[\S 4.4]{Tohoku}, and the vanishing follows from the existence of an isomorphism of $G$-modules 
$(\iota^*p_*\Z)_x\simeq\Z[G]^k$ for some $k\geq 0$,
as $p(T(\R))\cap\Theta=\varnothing$.
 Consider the commutative diagram in $D^+(\bS^1)$:
\begin{equation}
\label{truncdiag}
\begin{aligned}
\xymatrix
@R=0.4cm 
@C=0.5cm 
{
\tau_{\geq 1}\tau_{\leq 2}\RHom_{G}(\Z,\iota^*\sQ(1)) \ar^{\sim}[r]\ar^{\phi}[d]&\tau_{\geq 1}\tau_{\leq 2}\RHom_{G}(\Z,\Z(1)[1])  \\
\tau_{\geq 1}\tau_{\leq 2}\RHom_{G}(\Z,\Z/n(1)) \ar[r]\ar[ur]&\tau_{\geq 1}\tau_{\leq 2}\RHom_{G}(\Z,\Z/2),
}
\end{aligned}
\end{equation}
whose arrows are induced by (\ref{diagQ}), except for the lower horizontal arrow that is given by reduction modulo $2$, and where $\tau_{\geq 1}$ and $\tau_{\leq 2}$ are truncation functors. We have shown above that $\RHom_{G}(\Z,\iota^*p_*\Z)\in D^+(\bS^1)$ is concentrated in degree $0$; it thus follows from (\ref{diagQ}) that the upper horizontal arrow of (\ref{truncdiag}) is an isomorphism.

All complexes and arrows in (\ref{truncdiag}) may be computed using (\ref{cxcohoeq}), allowing to rewrite (\ref{truncdiag}) as follows:
\begin{equation}
\label{truncdiag2}
\begin{aligned}
\xymatrix
@R=0.4cm 
@C=0.5cm 
{
[\iota^*\sQ(1)/\langle 1-\sigma\rangle\xrightarrow{1+\sigma}(\iota^*\sQ(1))^G] \ar^(0.74){\sim}[r]\ar_{(\phi_2,\phi)}[d]&\Z/2[-2]  \\
\Z/2[-1]\oplus \Z/2[-2] \ar^{}[r]\ar_(.65){(0,1)}[ur]&\Z/2[-1]\oplus \Z/2[-2],
}
\end{aligned}
\end{equation}
where the upper left complex is concentrated in degrees $1$ and $2$. Consider the morphism $\Z/2[-2]\to\Z/2[-1]$ in $D^+(\bS^1)$ obtained by composing the inverse of the upper horizontal arrow of (\ref{truncdiag2}), the left vertical arrow of (\ref{truncdiag2}), and the projection to the factor $\Z/2[-1]$. It corresponds to an extension of $\Z/2$ by $\Z/2$ that is, in view of (\ref{truncdiag2}), given by (\ref{Gext}). Let us emphasize that this shows the exactness of  (\ref{Gext}).

Applying $H^2$ to any of the diagrams (\ref{truncdiag}) or (\ref{truncdiag2}) yields a commutative diagram:
\begin{equation}
\label{truncdiag3}
\begin{aligned}
\xymatrix
@R=0.4cm 
@C=0.5cm 
{
H^2_G(\bS^1,\iota^*\sQ(1)) \ar^(0.6){\sim}[r]\ar
[d]& H^0(\bS^1)  \\
H^1(\bS^1)\oplus H^0(\bS^1) \ar^{}[r]\ar_(.65){(0,1)}[ur]&H^1(\bS^1)\oplus H^0(\bS^1),
}
\end{aligned}
\end{equation}
where we have used that $\bS^1$ has cohomological dimension $1$ to identify the upper left group with $H^2_G(\bS^1,\iota^*\sQ(1))$. We now study the class $\iota^*\zeta\in H^2_G(\bS^1,\iota^*\sQ(1))$. 

If $x\in\bS^1$, one has $H^2_G(x,\Z(1))=H^2(G,\Z(1))=0$, so that $\eta|_x=0$. We deduce from Lemma \ref{etazeta} that $\phi(\zeta)|_x=-\talpha|_x\neq 0\in H^2_G(x,\Z/n(1))$, where the non-vanishing follows from the definition of $\Theta$ and (\ref{Brlift}). As a consequence, the image of $\iota^*\zeta$ in $H^0(\bS^1)$ by the left vertical arrow of (\ref{truncdiag3}), or equivalently by the upper horizontal arrow of (\ref{truncdiag3}), is equal to $1$. It follows that the image of $\iota^*\zeta$ in $H^1(\bS^1)$ by the left vertical arrow of (\ref{truncdiag3}) is the image of $1\in H^0(\bS^1)$ by the boundary map of (\ref{Gext}). Since the morphism $\Z/2[-1]\to\Z/2[-1]$ in the lower horizontal arrow of (\ref{truncdiag2}) is the identity, this class also equals the image of $\iota^*\zeta$ in the factor $H^1(\bS^1)$ of the right bottom group of (\ref{truncdiag3}), which is exactly $\iota^*[\phi(\zeta)]_1$, proving the lemma.
\end{proof}

It follows from Lemma \ref{Gzeta} that the sheaf $\sG$ is locally constant on $\SI$, as an extension of locally constant sheaves. One can compute that its stalks are isomorphic to $\Z/2\oplus\Z/2$ if $4\mid n$ (resp. to $\Z/4$ if $4\nmid n$), but we will not need this fact. 

 It is easy to describe a double cover $\trho:\tSI\to \bS^1$ whose class $\tep\in H^1(\bS^1)$ is the image of $1\in H^0(\bS^1)$ by the boundary map of (\ref{Gext}): define $\tSI$ to be the subset of the \'etal\'e space of $\sG$ whose fiber over $x\in \bS^1$ consists of the elements in $\sG_x$ whose image by $(0,\phi)$ in (\ref{Gext}) is equal to $1\in \Z/2$.

Recall from \S\ref{alternating} the construction of the finite \'etale double cover $\hp:\hV\to V$. Pulling back the covering $\hp:\hV(\C)\to V(\C)$ by $\iota:\bS^1\to V(\C)$
yields a double cover $\hrho:\hSI\to\SI$.
We denote its class by $\hep\in H^1(\bS^1)$.

\begin{lem}
\label{algprecise}
One has the identity:
$$\tep=\hep +\sum_{4\nmid m_i}\iota^*\cl_{\R}(\Delta_i)\in H^1(\bS^1).$$
\end{lem}

\begin{proof}
 Define $W:=\{x\in\bS^1\mid \iota(x)\notin\Delta(\R)\}$: it is the complement of finitely many points in $\bS^1$. Our proof has two steps. First, we construct an isomorphism $\chi:\hW\xrightarrow{\sim}\tW$ between the double covers $\hrho:\hW\to W$ and $\trho:\tW\to W$ of $W$ obtained by restricting $\hrho$ and $\trho$ to $W$. Second, we fix $x\in \bS^1$ with $\iota(x)\in \Delta_i(\R)$, and prove that the isomorphism $\chi$ extends through $x$ if and only if $4\mid m_i$. The lemma follows.

Let $x\in W$, and let $(a_1,\dots,a_{n/2},b_1,\dots,b_{n/2})$ be an ordering of $p^{-1}(x)\subset T(\C)$ defining an element $y\in\hrho^{-1}(x)$. Up to reordering using an even permutation, we may assume that the complex conjugation acts by $\sigma(a_j)=b_j$. The function $f:p^{-1}(x)\to \Z(1)$ defined by $f(a_j)=\sqrt{-1}$ and $f(b_j)=0$, viewed as an element in $(p_*\Z(1))_x$, 
induces an element in $\sQ(1)_x$ that one verifies to be $G$-invariant. As a consequence, $(0,f)\in \Z/2\oplus\sQ(1)^G_x$ induces, via (\ref{defG}), an element $z\in\sG_x$. The image $\phi(f)$ of $z$ by the right arrow  of  (\ref{Gext}) is the non-zero element of $\Z/n(1)^G\simeq \Z/2$. Thus, $z$ may be viewed as an element in $\trho^{-1}(x)$. It is a verification to check that changing the ordering of $p^{-1}(x)$ by a permutation changes the element  $z\in\trho^{-1}(x)$ if and only if the permutation is odd. The assignment $y\mapsto z$ thus induces a well-defined canonical bijection $\chi_x:\hrho^{-1}(x)\to \trho^{-1}(x)$, giving rise to a canonical isomorphism 
$\chi:\hW\xrightarrow{\sim}\tW$, and completing the first step of the proof.

We proceed to the second step. Fix $x\in \bS^1\setminus W$, and let $\Delta_i\subset \Delta$ be the component such that $\iota(x)\in \Delta_i(\R)$.
Let $(z_1,z_2)\in\cO_{S,\iota(x)}$ be a local system of parameters at $\iota(x)$ such that $z_1$ is a local equation of $\Delta$ at $\iota(x)$.
The rational map $(z_1,z_2):S(\C)\dashrightarrow \C^2$ is a local diffeomorphism at $\iota(x)$. Pulling back a small enough ball centered at $(0,0)\in\C^2$ yields a $G$-stable contractible neighbourhood $\Lambda$ of $\iota(x)$ in $ V(\C)\subset S(\C)$, isomorphic to the unit ball in $\C^2$ with coordinates $(z_1,z_2)$, such that $\Delta(\C)\cap\Lambda$ is defined in $\Lambda$ by the equation $z_1=0$, and on which $\sigma$ acts by $(z_1,z_2)\mapsto (\bar{z}_1,\bar{z}_2)$.
 Let $\Lambda_{\bS^1}$ be a contractible neighbourhood of $x$ in $\bS^1$ such that $\iota(\Lambda_{\bS^1})\subset \Lambda$ and such that $\Lambda_{\bS^1}\cap\iota^{-1}(\Delta(\C))=\{x\}$. Let $x^+, x^-\in \Lambda_{\bS^1}$ be two points in the two connected components of $\Lambda_{\bS^1}\setminus\{x\}$. Let $c:[0,1]\to \Lambda\setminus(\Delta(\C)\cap\Lambda)$ be a continuous path joining $\iota(x^+)$ and $\iota(x^-)$, and let $\bar{c}:[0,1]\to \Lambda\setminus(\Delta(\C)\cap\Lambda)$ be defined by $\bar{c}(t)=\sigma(c(t))$. Our system of coordinates in $\Lambda$ makes it clear that one may choose $c$ so that the loop $c^{-1}\bar{c}$
is a generator of $\pi_1(\Lambda\setminus(\Delta(\C)\cap\Lambda),\iota(x^-))\simeq\Z$.

Consider the diagram of bijections of sets with two elements:
\begin{equation}
\begin{aligned}
\label{diagbij}
\xymatrix
@R=0.3cm 
{
\hrho^{-1}(x^+)\ar[r]_{\sim}^{\chi_{x^+}}\ar_{\hu}^{\wr}[d]&  \trho^{-1}(x^+)\ar^{\wr}_{\tu}[d]\\
\hrho^{-1}(x^-) \ar_{\sim}^{\chi_{x^-}}[r]&  \trho^{-1}(x^-)
}
\end{aligned}
\end{equation}
whose vertical isomorphisms stem from the unique trivializations of $\hrho$ and $\trho$ on $\Lambda_{\bS^1}$. Our goal is to understand when (\ref{diagbij}) commutes. Since $p^{-1}(x)\subset T(\C)$ contains no real points, no point $a\in p^{-1}(x^+)$ belongs to the same orbit as $\sigma(a)$ under the monodromy action of $\pi_1(\Lambda\setminus(\Delta(\C)\cap\Lambda),\iota(x^+))\simeq \Z$. It follows that there exists an ordering $(a_1^+,\dots,a_{n/2}^+,b_1^+,\dots,b_{n/2}^+)$ of 
$p^{-1}(x^+)\subset T(\C)$ with the property that $\{a_1^+,\dots, a_{n/2}^+\}$ is stable under the monodromy action of 
$\pi_1(\Lambda\setminus(\Delta(\C)\cap\Lambda),\iota(x^+))$, and that $\sigma(a^+_j)=b^+_{j}$. Let $y^+\in \hrho^{-1}(x^+)$ be the point thus defined. Since the map $p:T(\C)\to S(\C)$  is unramified over the image of $c$, one may lift $c$ to a path in $T(\C)$ going from $a_j^+$ (resp. $b_j^+$) to a point that we denote by $a_j^-$ (resp. $b_j^-$). The ordering 
$(a_1^-,\dots,a_{n/2}^-,b_1^-,\dots,b_{n/2}^-)$ of $p^{-1}(x^-)\subset T(\C)$ represents $\hu(y^+)\in \hrho^{-1}(x^-)$ because $\hp:\hS(\C)\to S(\C)$ is unramified over the contractible set $\Lambda\subset V(\C)$. 

The sets $\{a_1^-,\dots,a_{n/2}^-\}$ and  $\{b_1^-,\dots,b_{n/2}^-\}$ are stable under the monodromy action of $\pi_1(\Lambda\setminus(\Delta(\C)\cap\Lambda),\iota(x^-))\simeq \Z$ on $p^{-1}(x^-)$, by our choice of $\{a_1^+,\dots, a_{n/2}^+\}$. The generator $c^{-1}\bar{c}$ acts on $\{a_1^-,\dots,a_{n/2}^-\}$ and $\{b_1^-,\dots,b_{n/2}^-\}$ by the permutations $(a_1^-,\dots,a_{n/2}^-)\mapsto (\sigma(b_1^-),\dots,\sigma(b_{n/2}^-))$ and $(b_1^-,\dots,b_{n/2}^-)\mapsto (\sigma(a_1^-),\dots,\sigma(a_{n/2}^-))$,
since $\sigma(a_j^+)=b_j^+$. These descriptions show that these two permutations have the same decompositions as products of cycles. We deduce that they are even if $4\mid m_i$ and odd otherwise. The ordering $(a_1^-,\dots,a_{n/2}^-,\sigma(a_1^-),\dots,\sigma(a_{n/2}^-))$ of $p^{-1}(x^-)$ thus induces a point $y^-\in \hrho^{-1}(x^-)$ that is equal to $\hu(y^+)$ if and only if $4\mid m_i$.

Since $\{a_1^+,\dots, a_{n/2}^+\}$ is stable under the monodromy, there exists a continuous function on $p^{-1}(\Lambda)\subset T(\C)$ that is equal to $\sqrt{-1}$ on the $a_j^+$ and on the $a_j^-$, and that is equal to $0$ on the $b_j^+$ and the $b_j^-$. Viewed as section in $H^0(\Lambda_{\bS^1},\iota^*p_*\Z(1))$, it induces a section $g\in H^0(\Lambda_{\bS^1},(\iota^*\sQ(1))^G)$. The section of $H^0(\Lambda_{\bS^1},\sG(1))$  induced by $(0,g)$ in  (\ref{defG}) has stalk $\chi_{x^+}(y^+)$ at $x^+$ and $\chi_{x^-}(y^-)$ at $x^-$, certifying that $\tu(\chi_{x^+}(y^+))=\chi_{x^-}(y^-)$.  We have proven that  (\ref{diagbij}) commutes if and only if $4\mid m_i$, thus completing the second step of the proof.
\end{proof}

It is now possible to prove Proposition \ref{algebraicity}.

\begin{proof}[Proof of Proposition \ref{algebraicity}]
Let $\he\in H^1_G(V(\C),\Z/2)$ be the class associated to the finite \'etale double cover $\hp:\hV\to V$ as in \S\ref{doublesec}. 
By Lemma \ref{algdouble} and by the surjectivity of the restriction map $\Pic(S)\to\Pic(V)$, there exists a line bundle $\varphi\in \Pic(S)$ such that $[\he]_1=\cl_{\R}(\varphi)|_{V(\R)}\in H^1(V(\R))$. We are going to prove the identity
\begin{equation}
\label{identityprecise}
([\phi(\zeta)]_1)|_{\Theta^0}=\Big(\cl_{\R}(\varphi)+\sum_{4\nmid m_i}\cl_{\R}(\Delta_i)\Big)|_{\Theta^0}\in H^1(\Theta^0),
\end{equation}
that implies the proposition.
The group $H_1(\Theta^0,\Z/2)$ is generated by classes of $\co$ loops $\bS^1\to \Theta^0$. One can in fact restrict to classes of $\ci$ embeddings
 $\iota:\bS^1\to\Theta^0$ meeting the ramification locus of $p$ transversally at general points (combine $\ci$ approximation,
 transversality theorems,
and replace a loop by a union of loops to remove self-intersections). By duality, it suffices to prove that
\begin{equation}
\label{identityprecise2}
\iota^*[\phi(\zeta)]_1=\iota^*\cl_{\R}(\varphi)+\sum_{4\nmid m_i}\iota^*\cl_{\R}(\Delta_i)\in H^1(\bS^1),
\end{equation}
for any such embedding $\iota$. The commutativity of the diagram
\begin{equation*}
\begin{aligned}
\xymatrix
@R=0.3cm 
@C=0.5cm 
{
H^1_G(V(\C),\Z/2)\ar[r]\ar[d]&H^1_G(V(\R),\Z/2)\ar[d]  \\
H^1(V(\C),\Z/2)\ar[r]&H^1(V(\R),\Z/2)\ar^(.65){\iota^*}[r]&H^1(\bS^1)
}
\end{aligned}
\end{equation*}
shows that $\hep=\iota^*[\he]_1\in H^1(\bS^1)$, so that $\hep=\iota^*\cl_{\R}(\varphi)\in H^1(\bS^1)$. The identity (\ref{identityprecise2}) then follows from Lemmas \ref{Gzeta} and \ref{algprecise}, proving the proposition.
\end{proof}

\subsection{The obstruction}
\label{subobstr}

We keep the notation of \S\S\ref{toposubsec}--\ref{algebraicitypar}.

\begin{proof}[Proof of Proposition \ref{topoobstr}]
By Lemma \ref{etazetauseful}, the reduction of the class $p^*\eta+\psi(\zeta)\in H^2_G(T_{U^0}^{}(\C),\Z(1))$ modulo $n$ is $p^*\talpha|_{U^0}$.
 Since $p^*\alpha|_{U^0}\in\Br(T^{}_{U^0})\subset \Br(\R(T))$ vani\-shes by hypothesis, it follows from (\ref{Brdeco}) that there exists $\gamma\in H^2_G(T_{U^0}^{}(\C),\Z(1))$ and $\varphi\in \Pic(T_{U_0})$ such that:
\begin{equation}
\label{equalityupstairs}
p^*\eta+\psi(\zeta)=n\gamma+\cl(\varphi)\in H^2_G(T^{}_{U^0}(\C),\Z(1)).
\end{equation}
The class $p_*\psi(\zeta)$ vanishes because $\Tr\circ\psi=0$. Pushing forward (\ref{equalityupstairs}) by $p$ yields:
\begin{equation}
\label{equalitydownstairs}
n(\eta-p_*\gamma)=\cl(p_*\varphi)\in H^2_G(U^0(\C),\Z(1)).
\end{equation}
Since the cokernel of Krasnov's cycle class map $\cl:\Pic(U^0)\to H^2_G(U^0(\C),\Z(1))$ is torsion-free \cite[Proposition 2.9]{BWI}, and since the restriction map $\Pic(S)\to\Pic(U^0)$ is surjective, there exists $\mu\in \Pic(S)$ such that $\eta-p_*\gamma=\cl(\mu|_{U^0})\in H^2_G(U^0(\C),\Z(1))$. 
Applying (\ref{krasnovcompa}) shows that
\begin{equation}
\label{eqrestreinte}[\eta]_1-[p_*\gamma]_1=\cl_{\R}(\mu|_{U^0})\in H^1(U^0(\R)).
\end{equation}
By
 \cite[Proposition 1.22]{BWI}, $[p_*\gamma]_1=p_*([\gamma]_1)$. Since $p(T(\R))\cap\Theta=\varnothing$,
one has $(p_*([\gamma]_1))|_{\Theta^0}=0\in H^1(\Theta^0)$. By (\ref{eqrestreinte}), we see that $([\eta]_1)|_{\Theta^0}=\cl_{\R}(\mu)|_{\Theta^0}\in H^1(\Theta^0)$.
By Proposition \ref{algebraicity}, $([\phi(\zeta)]_1)|_{\Theta^0}=\cl_{\R}(\theta)|_{\Theta^0}\in H^1(\Theta^0)$ for some $\theta\in \Pic(S)$. By definition of $\zeta$ and $\eta$, $([\xi]_1)|_{\Theta^0}=\cl_{\R}(\nu)|_{\Theta^0}\in H^1(\Theta^0)$ for $\nu=\theta+\mu$.
Since $U\setminus U^0$ is finite, the map $H^1(\Theta)\hookrightarrow H^1(\Theta^0)$ is injective, and $([\xi]_1)|_{\Theta}=\cl_{\R}(\nu)|_{\Theta}\in H^1(\Theta)$.
\end{proof}

We may now complete the proof of Theorem \ref{pinr}.

\begin{proof}[Proof of Theorem \ref{pinr}]
The first statement is a consequence of Proposition \ref{proofarbitraryperiod}.

To prove the second statement, let $S$ be a connected smooth projective surface over $\R$, and let $\alpha\in\Br(S)\subset\Br(\R(S))$ be a class of even period $n$. By de Jong's theorem \cite{deJong} and a norm argument, $\ind(\alpha)$ is equal to $n$ or $2n$. Suppose that $\ind(\alpha)=n$. Then there exists a degree $n$ extension $L/\R(S)$
such that $\alpha_{L}=0$. Let $T$ be a connected smooth projective surface over $\R$ with function field $L$ such that $\R(S)\subset L$ is induced by a morphism $p:T\to S$. 
Let $\talpha\in H^2_G(S(\C),\Z/n(1))$ be a lift of $\alpha$ in (\ref{Brlift}). The reduction $\xi\in H^2_G(S(\C),\Z/2)$ of $\talpha$ modulo $2$ is a lift of $\frac{n}{2}\alpha\in\Br(S)[2]$ in (\ref{Brlift}).
 Proposition \ref{topoobstr} shows the existence of $\nu\in\Pic(S)$
such that  $([\xi]_1)|_\Theta=\cl_{\R}(\nu)|_\Theta \in H^1(\Theta)$. Let $\xi'\in H^2_G(S(\C),\Z/2)$ be the difference between $\xi$ and the reduction of $\cl(\nu)$ modulo $2$. It is another lift of $\frac{n}{2}\alpha\in \Br(S)[2]$ in (\ref{Brlift}), and satisfies $([\xi']_1)|_\Theta=0\in H^1(\Theta)$ by (\ref{krasnovcompa}). This concludes the proof.
\end{proof}

\section{Examples}
\label{secex}

We now illustrate the real period-index problem with a few examples.

\subsection{Real Enriques surfaces}
\label{EnriquesR}

In \S\ref{EnriquesR}, we describe which Enriques surfaces $S$ over $\R$ carry Brauer classes $\alpha\in\Br(S)$ with $\per(\alpha)\neq\ind(\alpha)$, thus proving Theorem \ref{thEnr}. 

The geometry of real Enriques surfaces $S$ is well understood. The Brauer group $\Br(S)$ and the image of the Borel-Haefliger map $\cl_{\R}:\Pic(S)\to H^1(S(\R))$ have been computed by Mangolte and van Hamel \cite[Theorems 1.3 and 4.4]{MvH} (see also \cite[Remark 3.18 (ii)]{BWI}), the Witt group $W(S)$ has been computed by Sujatha and van Hamel \cite[Theorems 2.6 and 3.3]{SvH}, and their possible topological types  have been classified by Degtyarev, Itenberg and Kharlamov \cite{DIK}.

We will rely on Proposition \ref{duality} below,
that is an application of the duality theorem \cite[Theorem 1.12]{BWI} proven in a joint work with Wittenberg.
If $X$ is a connected smooth projective variety of dimension $d$ over $\R$, we let
 $\deg: H^d(X(\R))\to\Z/2$ be the degree map. We also denote by $\deg: H^*(X(\R))\to \Z/2$ the map constructed as the composition of the projection on $H^d(X(\R))$ and of the degree map.

\begin{prop}
\label{duality}
Let $S$ be a connected smooth projective surface over $\R$ and let $w=(w_i)\in H^*(S(\R))$ be the total Stiefel--Whitney class of $S(\R)$.

Then the image of the composition
$$\rho:H^2_G(S(\C),\Z/2)\to H^2_G(S(\R),\Z/2)\xrightarrow{\sim} H^0\soplus H^1\soplus H^2(S(\R))\to H^0\soplus H^1(S(\R))$$
of the restriction to the real locus, of the canonical decomposition (\ref{can2}) and of the projection is
the set of $(a_0,a_1)\in H^0\soplus H^1(S(\R))$ such that, for every class $\delta\in H^1_G(S(\C),\Z/2)$ with  $[\delta]_0=0$, one has $\deg([\delta]_1\cupp(a_1+a_0\cupp w_1))=0$.
\end{prop}

\begin{proof}
Consider the diagram:
\begin{equation}
\label{diagEnr}
\begin{aligned}
\xymatrix
@R=0.3cm 
@C=0.22cm 
{
0\ar[r]&H^2_{G,S(\R)}(S(\C),\Z/2)\ar[r]\ar@{=}[d]&H^2_{G}(S(\C),\Z/2)\ar[r]\ar[d]&H^2_{G}(S(\C)\ssetminus S(\R),\Z/2)\ar@/_01.0pc/[l]_{\tau}\ar[r]&0  \\
&H^0(S(\R))\ar^(.388){\cupp w}[r]&H^0\soplus H^1\soplus H^2(S(\R))\ar[d]&\\
&&H^0\soplus H^1(S(\R))&
}
\end{aligned}
\end{equation}
whose top row stems from the long exact sequence of cohomology with support and is exact and canonically split by \cite[Proposition 1.3]{BWI}, whose left vertical identification is given by equivariant purity \cite[(1.20)]{BWI}, and whose middle column is the morphism $\rho$. The horizontal arrow making the diagram commute
is the cup-product by the class $\gamma\in H^*(S(\R))$ of \cite[Definition 1.4]{BWI}, that is equal to the total Stiefel-Whitney class $w$ by \cite[Remark 1.6 (i)]{BWI}. The canonical section $\tau$ of the top row of (\ref{diagEnr}) described in the proof of \cite[\S 1.3.1]{BWI} induces a decomposition 
\begin{equation}
\label{decdual}
H^2_{G}(S(\C),\Z/2)=H^2_{G,S(\R)}(S(\C),\Z/2)\oplus H^2_{G}(S(\C)\setminus S(\R),\Z/2).
\end{equation}
We now compute separately the image by $\rho$ of the two factors of (\ref{decdual}). 

A glance at diagram (\ref{diagEnr}) shows that the image of the factor $H^2_{G,S(\R)}(S(\C),\Z/2)$ by $\rho$ is equal to the set of $(a_0,a_1)\in H^0\oplus H^1(S(\R))$ such that $a_1=a_0\cupp w_1$.

The image by $\rho$ of the factor $H^2_{G}(S(\C)\setminus S(\R),\Z/2)$ consists of classes of the form $(0,a_1)\in H^0\oplus H^1(S(\R))$, in view of the construction of $\tau$ given in \cite[\S 1.3.1]{BWI}.
 We now describe what are the classes $a_1\in H^1(S(\R))$ that appear. The duality theorem \cite[Theorem 1.12]{BWI} (more precisely, the duality between the images of the maps denoted by $w_1$ and $w_2$ in \emph{loc.\ cit.}) shows, after unravelling  definitions, that
the image of the composition
$$H^2_{G}(S(\C)\ssetminus S(\R),\Z/2)\xrightarrow{\tau}H^2_{G}(S(\C),\Z/2)\to H^0\soplus H^1\soplus H^2(S(\R))\to H^1\soplus H^2(S(\R))$$
of $\tau$, of the restriction to $S(\R)$ and of (\ref{can2}), and of the projection, is dual to the image of the morphism $H^1_G(S(\C),\Z/2)\to H^0\oplus H^1(S(\R))$ given by the canonical decomposition (\ref{can2})  with respect to the natural pairing $(x,y)\mapsto\deg(x\cupp y)$. We deduce that the classes $a_1\in H^1(S(\R))$ that appear are exactly those that are orthogonal to $ [\delta]_1\in H^1(S(\R))$ for every $\delta\in H^1_G(S(\C),\Z/2)$ such that $[\delta]_0=0$.

Combining these two computations gives a complete description of the image of $\rho$, and proves the proposition.
\end{proof}

To prove Theorem \ref{thEnr}, we combine Proposition \ref{duality} in the case of an Enriques surface $S$ and \cite[Theorem 4.4]{MvH}. We still denote by $w=(w_i)\in H^*(S(\R))$ the total Stiefel-Whitney class of $S(\R)$.

\begin{proof}[Proof of Theorem \ref{thEnr}]
If $S(\R)=\varnothing$, the equality $\per(\alpha)=\ind(\alpha)$ holds for every $\alpha\in \Br(S)$ by Theorem \ref{pinr}. From now on, we suppose that $S(\R)\neq\varnothing$.

By (\ref{BrauerC}), one has
$\Br(S_{\C})\xrightarrow{\sim} H^3(S(\C),\Z)_{\tors}=\Z/2$  since $H^2(S,\cO_S)=0$. It follows that $\Br(S)$ is $4$-torsion (see also the more precise \cite[Theorem 1.3]{MvH}). If $\alpha\in\Br(S)$ has period $4$, $\alpha_{\C}\in \Br(S_{\C})$ has period $2$, hence index $2$ by de Jong's theorem \cite{deJong}. We deduce that $\alpha$ has index $4$. It remains to study classes $\alpha\in\Br(S)$ of period $2$.

Consider the diagram:
\begin{equation}
\label{BrEnr}
\begin{aligned}
\xymatrix
@R=0.3cm 
@C=0.5cm 
{
0\ar[r]&\Pic(S)/2\ar[r]\ar_{\cl_{\R}}[rd]& H^2_G(S(\C),\Z/2)\ar[d]\ar[r]&\Br(S)[2]\ar[r]&0  \\
&&H^1(S(\R))&
}
\end{aligned}
\end{equation}
whose top row is (\ref{Brlift}), whose vertical arrow is $\xi\mapsto[\xi]_1$, that commutes by (\ref{krasnovcompa}).
By \cite[Theorem 4.4]{MvH}, the image of $\cl_{\R}:\Pic(S)\to H^1(S(\R))$ is the orthogonal of $w_1\in H^1(S(\R))$. In particular, if $S(\R)$ is orientable, $\cl_{\R}$ is surjective \cite[Theorem 1.1]{MvH} and it follows from (\ref{BrEnr}) that every class $\alpha\in \Br(S)[2]$ of period $2$ has a lift $\talpha\in H^2_G(S(\C),\Z/2)$ with $[\talpha]_1=0$. Theorem \ref{pinr} then implies that $\ind(\alpha)=2$, proving the theorem in this case.  From now on, assume that $S(\R)$ is not orientable, i.e. that $w_1\neq 0$. By Poincar\'e duality, this implies that $\cl_{\R}$ is not surjective.

 The isomorphism
$H^1_G(S(\C),\Z/2)\simeq H^1_{\et}(S,\Z/2)$ between equivariant Betti cohomology and \'etale cohomology \cite[Corollary 15.3.1]{Scheiderer} and the Kummer exact sequence give a short exact sequence:
$$0\to\R^*/\R^{*2}\to H^1_G(S(\C),\Z/2)\simeq H^1_{\et}(S,\Z/2)\to\Pic(S)[2]\to 0.$$
The group $\Pic(S)[2]$ is isomorphic to $\Z/2$, and is generated by the canonical bundle $K_S$. The two preimages of $K_S$ in $H^1_{\et}(S,\Z/2)$ are the two finite \'etale covers of $S$ considered in \cite[\S 1.3]{Halves}, giving rise to the two halves of $S$. It follows that there exists a non-zero class $\delta\in H^1_G(S(\C),\Z/2)$ with  $[\delta]_0=0$ if and only if one of the two halves of $S$ is empty, and that this class is unique.

If the two halves of $S$ are nonempty, there exists no such class $\delta\in H^1_G(S(\C),\Z/2)$. By Proposition \ref{duality}, there exists $\talpha\in H^2_G(S(\C),\Z/2)$ such that $[\talpha]_0=1$ and $[\talpha]_1$ does not belong to the image of $\cl_{\R}$. The induced class $\alpha\in\Br(S)[2]$ does not have index $2$ by Theorem \ref{pinr} and (\ref{BrEnr}), proving the theorem in this case.

Suppose from now on that $S$ has exactly one nonempty half. Then there is a unique $\delta\in H^1_G(S(\C),\Z/2)$ with  $[\delta]_0=0$. By Lemma \ref{algdouble} and \cite[Th\'eor\` eme 4]{Kahn}, one has $[\delta]_1=\cl_{\R}(K_S)=w_1$. 
Note that if $\Sigma$ is a connected component of $S(\R)$, one has $\congru{w_1(\Sigma)^2}{\chi(\Sigma)}{2}$ by classification of compact $\ci$ surfaces.

Assume first that $S(\R)$ has an odd number of connected components with odd Euler characteristic. Equivalently, $\deg(w_1^2)\neq 0$. By Proposition \ref{duality}, there exists $\talpha\in H^2_G(S(\C),\Z/2)$ such that  $[\talpha]_0=1$ and $[\talpha]_1=w_1$. Since $w_1$ is not in the image of $\cl_{\R}$ because $\deg(w_1^2)\neq 0$,
Theorem \ref{pinr} and (\ref{BrEnr}) show that the Brauer class $\alpha\in\Br(S)[2]$ induced by $\talpha$ had index $4$.

To conclude, assume that $S(\R)$ has an even number of connected components with odd Euler characteristic. We choose a class $\alpha\in\Br(S)[2]$ of period $2$, and a lift $\talpha\in H^2_G(S(\C),\Z/2)$ of $\alpha$ in (\ref{BrEnr}).
By Proposition \ref{duality}, $\deg([\talpha]_1\cupp w_1\splus[\talpha]_0\cupp w_1^2)\s=0$. If $\deg([\talpha]_1\cupp w_1)=0$, define $a:=[\talpha]_1\in H^1(S(\R))$. Otherwise, $\deg([\talpha]_0\cupp w_1^2)\neq 0$ so that there exists a connected component $\Sigma\subset S(\R)$ of odd Euler characteristic such that $([\talpha]_0)|_{\Sigma}= 0$. In this case, define $a\in H^1(S(\R))$ by $a|_{S(\R)\setminus \Sigma}=([\talpha]_1)|_{S(\R)\setminus \Sigma}$ and $a|_\Sigma=([\talpha]_1)|_{\Sigma}+w_1(\Sigma)$. In both cases, $\deg(a\cupp w_1)=0$, so that there exists $\theta\in \Pic(S)$ such that $\cl_{\R}(\theta)=a$ by  \cite[Theorem 4.4]{MvH}. Modifying $\talpha$ by the image of $\theta$ in (\ref{BrEnr}) and applying Theorem \ref{pinr} shows that $\ind(\alpha)=2$ and concludes.
\end{proof}
\subsection{A K3 surface over a non-archimedean real closed field}
\label{K3nonarch}
We now prove Proposition \ref{propK3nonarch}, thus showing that Theorem \ref{pinr} fails over general real closed fields such as $\bK:=\cup_n \R((t^{1/n}))$. The surface $S$ we use is exactly that of \cite[Example 15.2.2]{BCR}. It has the property that no rational function on $S$ is positive on one of the semi-algebraic connected components of $S(\bK)$ and negative on the others. 

\begin{proof}[Proof of Proposition \ref{propK3nonarch}]
Let $[u:v:w]$ be coordinates on $\bP^2_{\bK}$, where $\bK$ is the real closed field $\cup_n \R((t^{1/n}))$. Consider the sextic double cover $\cX$ of $\bP^2_{\R[[t]]}$ defined by:
$$\cX:=\{z^2=(w^2-u^2-v^2)u^2v^2-tw^6\}.$$
Since $\cX_{\bK}$ has rational double points as singularities, its minimal resolution of singularities  is a K3 surface $S$ over $\bK$. One checks that $S(\bK)$ has four semi-algebraic connected components separated by the signs of $u/w$ and $v/w$, that are semi-algebraically isomorphic to spheres. In particular, $H^1(S(\bK),\Z/2)=0$. 

Let $\Xi\subset S(\bK)$ be the connected component such that $u/w,v/w>0$.
By \cite[Proposition 3.1.2]{CTParimala} (see also \cite[Theorem 2.8]{vanHamel}), there exists $\alpha\in\Br(S)[2]$ such that $\alpha$ is trivial in restriction to $x\in S(\bK)$ if and only if $x\in\Xi$. Suppose for contradiction that $\alpha$ has index $2$. Then it is a quaternion class: there exist $f,g\in\bK(S)^*$ such that $\alpha=(f,g)\in \Br(\R(S))$. In particular, if $x\in S(\bK)$ lies outside of the zeros and poles of $f$ and $g$,  one has $x\in \Xi$ if and only at least one of $f(x)$ and $g(x)$ is positive.

Let $n$ be such that $f,g \in\R((t^{1/n}))(\cX_{\R((t^{1/n}))})^*$.
Multiplying $f$ and $g$ by an appropriate power of $t^{1/n}$, we may specialize them to rational functions $f_0,g_0\in\R(\cX_{\R})^*$ on the special fiber $\cX_{\R}$ of $\cX$.
Let $Q$ be the normalization of $\cX_{\R}$: it is the quadric $Q=\{z^2=w^2-u^2-v^2\}\subset\bP^3_{\R}$ over $\R$. View $f_0,g_0$ as rational functions on $Q$. By our choice of $f$ and $g$, if $x\in Q(\R)$ is such that $u(x),v(x)\neq 0$ and lies outside of the poles of $f_0$ and $g_0$, then at least one of $f_0(x)$ and $g_0(x)$ is positive if and only if $u/w(x),v/w(x)>0$.

Since the signs of $f_0$ and $g_0$ are constant in a neighbourhood of $[u:v:w:z]=[0:-1:1:0]\in Q(\R)$, we deduce that the orders of vanishing of $f_0$ and $g_0$ along the divisor $D:=\{u=0\}\subset Q$ are even. It follows that if $x\in Q(\R)$ is chosen such that $w\neq 0$, $u=0$ and $v/w(x)>0$, and such that $x$ does not belong to any divisor of poles of $f_0$ or $g_0$ distinct of $D$, then the signs of $f_0$ and $g_0$ are constant in a neighbourhood of $x\in Q(\R)$. This is the required contradiction.
\end{proof}

\bibliographystyle{plain}
\bibliography{realperiodindex}

\end{document}